\documentclass{amsart}
\usepackage{amsmath, amssymb,amscd}
\usepackage{color}
\usepackage[usenames,dvipsnames,svgnames,table]{xcolor}
\usepackage{soul}
\setstcolor{red}
\sethlcolor{yellow}

\input xy
\xyoption{all}

\newtheorem{theorem}{Theorem}[section]
\newtheorem{lemma}[theorem]{Lemma}
\newtheorem{corollary}[theorem]{Corollary}
\newtheorem{fact}[theorem]{Fact}
\newtheorem{proposition}[theorem]{Proposition}
\newtheorem{claim}[theorem]{Claim}
\newtheorem{assumption}[theorem]{Assumptions}

\theoremstyle{definition}
\newtheorem{example}[theorem]{Example}
\newtheorem{remark}[theorem]{Remark}
\newtheorem{definition}[theorem]{Definition}

\def\th{\operatorname{Th}}
\def\aut{\operatorname{Aut}}

\def\acl{\operatorname{acl}}
\def\eq{\operatorname{eq}}
\def\dcl{\operatorname{dcl}}
\def\tp{\operatorname{tp}}

\def\cb{\operatorname{Cb}}
\def\spec{\operatorname{Spec}}

\def\pr{\operatorname{pr}}
\def\id{\operatorname{id}}

\def\jet{\operatorname{Jet}}

\def\alg{\operatorname{alg}}

\def\ecdf{\mathcal D\operatorname{-CF}_0}
\def\trdeg{\operatorname{trdeg}}
\def\dlocus{\underline\cD\operatorname{-locus}}
\def\locus{\operatorname{locus}}

\newcommand{\Ga}{{\mathbb G}_{\operatorname a}}
\newcommand{\cM}{{\mathcal M}}
\newcommand{\cE}{{\mathcal E}}
\newcommand{\cD}{{\mathcal D}}
\newcommand{\cF}{{\mathcal F}}
\newcommand{\cG}{{\mathcal G}}
\newcommand{\cK}{{\mathcal K}}
\newcommand{\cL}{{\mathcal L}}

\newcommand{\UU}{{\mathbb U}}
\newcommand{\bbs}{{\mathbb S}}

\def\Ind#1#2{#1\setbox0=\hbox{$#1x$}\kern\wd0\hbox to 0pt{\hss$#1\mid$\hss}
\lower.9\ht0\hbox to 0pt{\hss$#1\smile$\hss}\kern\wd0}
\def\ind{\mathop{\mathpalette\Ind{}}}
\def\Notind#1#2{#1\setbox0=\hbox{$#1x$}\kern\wd0\hbox to 0pt{\mathchardef
\nn=12854\hss$#1\nn$\kern1.4\wd0\hss}\hbox to
0pt{\hss$#1\mid$\hss}\lower.9\ht0 \hbox to
0pt{\hss$#1\smile$\hss}\kern\wd0}
\def\nind{\mathop{\mathpalette\Notind{}}}

\begin{document}

\title[Fields with operators]{Model theory of fields with\\ free operators in characteristic zero}

\author{Rahim Moosa}
\address{Rahim Moosa\\
University of Waterloo\\
Department of Pure Mathematics\\
200 University Avenue West\\
Waterloo, Ontario \  N2L 3G1\\
Canada}
\email{rmoosa@uwaterloo.ca}

\thanks{R. Moosa was supported by an NSERC Discovery Grant.  T. Scanlon was partially supported
 by NSF grants FRG DMS-0854839 and DMS-1001556 }

\author{Thomas Scanlon}
\address{Thomas Scanlon\\
University of California, Berkeley\\
Department of Mathematics\\
Evans Hall\\
Berkeley, CA \ 94720-3480\\
USA}
\email{scanlon@math.berkeley.edu}

\date{August 27, 2013}

\subjclass[2000]{}

\begin{abstract}
Generalising and unifying the known theorems for difference and differential fields, it is shown that for every finite free algebra scheme $\cD$ over a field $A$ of characteristic
zero, the theory of $\cD$-fields has a model companion $\ecdf$ which is simple and
satisfies the Zilber dichotomy for finite-dimensional minimal types.
\end{abstract}

\maketitle

\tableofcontents
\vfill
\newpage

\section{Introduction}

The theories of differential and difference fields instantiate some
of the most sophisticated ideas and theorems in model theoretic stability
theory.  For example, the theory of differentially closed fields of characteristic zero,
$\operatorname{DCF}_0$, is an $\omega$-stable theory for which the full panoply of
geometric stability theory applies from the existence and uniqueness of prime models
(and, hence, of differential closures) to the theory of liaison groups (and, thus,
a very general differential Galois theory) to the Zilber trichotomy for strongly minimal sets
(from which strong theorems about function field arithmetic have been deduced).
Likewise, the model companion of the theory of difference fields, $\operatorname{ACFA}$,
is supersimple and admits an analogous theory of internal automorphism groups and
satisfies a version of the Zilber dichotomy for its minimal types.   Beyond the formal
analogies and parallel theorems, the proofs of the basic results in the model theory of
differential and difference fields follow similar though not identical lines.  In this
paper we formalise the sense in which these theories are specialisations of a common theory
of fields with operators and how the theories may be developed in one fell swoop. On the other
hand, features which emerge from the general theory explain how the theories of differential
and difference fields diverge.

By definition a derivation on a commutative ring $R$ is
an additive map $\partial:R \to R$ which satisfies the Leibniz rule $\partial(xy) = x \partial(y) +
y \partial(x)$.  Equivalently, the function $e:R \to R[\epsilon]/(\epsilon^2)$ given by $x \mapsto x + \partial(x) \epsilon$
is a homomorphism of rings.  An endomorphism $\sigma:R \to R$ of a ring is simply a ring homomorphism from the ring
$R$ back to itself, but at the risk of complicating the definition, we may also say that a function $\sigma:R \to R$ is
an endomorphism if the function $e:R \to R \times R$ given by $x \mapsto (x,\sigma(x))$ is a homomorphism of rings.
With each of the latter presentations we see differential (respectively, difference) ring as
a $\cD$-ring in the sense introduced in~\cite{paperA}.

As the details of the $\cD$-ring formalism along with  many examples are presented in Section~\ref{Dringsec},
we limit ourselves to a loose discussion here.  For each fixed ring scheme
$\cD$ (possibly over some base ring $A$) satisfying some additional requirements we have
a theory of $\cD$-fields.  In particular, we require that the underlying
additive group scheme of $\cD$ be some power of the additive group scheme so that for any $A$-algebra $R$,
$\cD(R) = (R^n,+,\boxtimes)$ where the multiplication $\boxtimes$ is given by some bilinear form defined over $A$.
We require that $\cD$ comes equipped with a functorial projection map to the standard ring scheme and that for the
sake of concreteness, read relative to coordinates this projection map be given by projection onto the
first coordinate.  A $\cD$-ring is then a pair $(R,e)$ consisting of an $A$-algebra $R$ and a map of $A$-algebras
$e:R \to \cD(R)$ which is a section of the projection.    In the motivating examples, $\cD(R) = R[\epsilon]/(\epsilon^2) = (R^2,+,\boxtimes)$
where $(x_1,x_2) \boxtimes (y_1,y_2) = (x_1 y_1, x_1 y_2 + x_2 y_1)$  ($\cD(R) = R \times R$ with coordinatewise ring operations,
respectively).

In general, using the coordinatization of $\cD(R)$, the data of a $\cD$-ring $(R,e)$ is equivalent to that of
a ring $R$ given together with a sequence $\partial_0, \ldots, \partial_{n-1}$ of operators $\partial_i:R \to R$
for which the map $e:R \to \cD(R)$ is given in coordinates by $x \mapsto (\partial_0(x), \ldots, \partial_{n-1}(x))$.
The requirements on such a sequence of operators that they define a $\cD$-ring structure may be expressed by certain
universal axioms.  For example, to say that the map $e:R \to \cD(R)$ is a section of the projection is just to say that
$(\forall x \in R) \partial_0(x) = x$ and the requirement that $e$ convert multiplication in $R$ to the multiplication of $\cD(R)$ may be expressed by insisting that certain polynomial relations hold amongst $\partial_0(x), \ldots, \partial_{n-1}(x);
\partial_0(y), \ldots, \partial_{n-1}(y); \partial_{0}(xy), \ldots, \partial_{n-1}(xy)$.  In this way, the class of $\cD$-fields is
easily seen to be first order in the language of rings augmented by unary function symbols for the operators
$\partial_0, \ldots, \partial_{n-1}$.  On the other hand, the interpretation of the operators as components of
a ring homomorphism permits us to apply ideas from commutative algebra and algebraic
geometry to analyze these theories.

Our first main theorem is that for any ring scheme $\cD$ (meeting the requirements set out in Section~\ref{Dringsec}), the
theory of $\cD$-fields of characteristic zero has a model companion, which we denote by $\ecdf$ and call
the theory of $\cD$-closed fields.  Our axiomatization of $\ecdf$ follows the geometric style which first appeared
in the Chatzidakis-Hrushovski axioms for ACFA~\cite{acfa1} and was then extended to differential fields by Pierce and
Pillay~\cite{piercepillay}.  Moreover, the proofs will be familiar to anyone who has worked through the corresponding
results for difference and differential fields.  Following the known proofs for difference and differential
fields, we establish a quantifier simplification theorem and show that $\ecdf$ is always simple.

As  noted above, the theory $\operatorname{DCF}_0$ is the quintessential $\omega$-stable theory, but
$\operatorname{ACFA}$, the model companion of the theory of difference fields is not even stable.   At a technical
level, the instability of $\operatorname{ACFA}$ may be traced to the failure of quantifier elimination which,
algebraically, is due to the non-uniqueness (up to isomorphism) of the extension of an automorphism of a field to
an automorphism of its algebraic closure.  We show that this phenomenon, namely that instability is tied to the
nonuniqueness of extensions of automorphisms, pervades the theory of $\cD$-fields.   That is, for each $\cD$
there is a finite list of associated endomorphisms expressible as linear combinations of the basic operators.
Since we require that $\partial_0$ is the identity map, one of these associated endomorphisms is always the identity
map.  If there are any others, then the theory of $\ecdf$ suffers from instability and the failure of
quantifier elimination just as does ACFA.  On the other hand, if there are no other associated endomorphisms, then
$\ecdf$ is stable.

The deepest of the fine structural theorems for types  in DCF$_0$ and in ACFA is the Zilber dichotomy for
minimal types, first established by Sokolovi\'{c} and Hrushovski for DCF$_0$ using Zariski geometries~\cite{hrushovskisokolovic}, for
ACFA$_0$ by Chatzidakis and Hrushovski through a study of ramification~\cite{acfa1}, and by Chatzidakis, Hrushovski and Peterzil
for ACFA in all characteristics using the theory of limit types and a refined form of the theory of Zariski geometries~\cite{acfa2}.
Subsequently, Pillay and Ziegler established a stronger form of the trichotomy theorem
in characteristic zero~\cite{pillayziegler03} by adapting
jet space arguments Campana and Fujiki used to study
complex manifolds~\cite{campana, fujiki}.  Here we implement the Pillay-Ziegler strategy
for $\ecdf$ by using the theory of $\cD$-jet spaces from~\cite{paperB}.  In particular, we show that finite dimensional types in $\ecdf$ satisfy the
canonical base property.

The model companion of the theory of difference fields of characteristic zero with $n$ automorphisms appears as
$\ecdf$ where $\cD(R) := R^{1+n}$ in contradistinction to the well-known fact that the theory of difference fields with
$n$ ($>1$) commuting automorphisms does not have a model companion.  On the other hand, if
$({\mathbb U}, \partial_0,\partial_1,\ldots,\partial_n) \models \cD\text{-CF}_0$ is sufficiently saturated, then the type definable field
obtained as the intersection of the fixed fields of all the elements of the commutator group of the group generated by
$\partial_1, \ldots, \partial_n$ has Lascar rank $\omega^n$ and may be regarded as a universal domain for difference
fields with $n$ commuting automorphisms. (See Section~1.2 of~\cite{HrMM} for a discussion of these issues.)
Likewise, models of the theory $\operatorname{DCF}_{0,n}$ of differentially
closed fields with $n$ commuting derivations may be realized as type definable fields in models of $\ecdf$ where
$\cD(R) = R[\epsilon_1,\ldots,\epsilon_n]/(\epsilon_1,\ldots,\epsilon_n)^2$.

While omitting commutation allows for model companions in characteristic zero, it complicates matters in positive
characteristic.  Under a natural algebraic hypothesis on $p$ and $\cD$, namely that there be some $\epsilon \in \cD(A)$
which is nilpotent but for which $\epsilon^p \neq 0$, we observe with Proposition~\ref{nomcp} that no model companion of the
theory of $\cD$-fields of characteristic $p$ exists.
This proposition is consonant with the known examples of $\operatorname{ACFA}_p$ and $\operatorname{DCF}_p$  where no such $\epsilon$ exists.
However, it implies that the theory of (not necessarily iterative) Hasse-Schmidt differential fields of positive characteristic does not have a model companion, which is at odds with the iterative theory $\operatorname{SCH}_{p,e}$ considered by Ziegler~\cite{ziegler03}.
While the theory of iterative $\cD$-fields developed in~\cite{paperB} was intended as an abstraction of the
theory of iterative Hasse-Schmidt differential fields, we have not yet understood the extent to which the theorems around SCH$_{p,e}$
generalise to iterative $\cD$-fields.  As such, we leave open the problems of which theories of iterative $\cD$-fields and which theories of
positive characteristic $\cD$-fields have model companions.

This paper is organized as follows.   We begin in Section~\ref{notion} with some remarks about our conventions.  With Section~\ref{Dringsec}
we recall the formalism of $\cD$-rings in detail and present several examples.  In Section~\ref{chapter-mc} we give axioms for
the theory $\ecdf$ and prove that it is in fact the model companion of the theory of $\cD$-fields of characteristic zero.
In Section~\ref{basicmodel} we establish the essential model theoretic properties of $\cD$-closed fields.  In Section~\ref{zilbersec}
we give a proof of the Zilber dichotomy for minimal types of finite dimension.
We conclude with an appendix in which we show that the theory of $\cD$-fields does not have a model companion
for most choices of $\cD$ in positive characteristic, and also explain how a convenient set of assumptions made early in the paper can be removed.

We are very grateful to the anonymous referee for a careful reading of an earlier version of this paper, and for suggesting changes that have lead to significant improvements.
We are also grateful to Omar Le\'on S\'anchez and Tamvana Makuluni for catching errors in an earlier version.

\section{Notation and conventions}
\label{notion}

All rings are commutative and unitary.
As a general rule, we follow standard conventions in model theory and differential algebra and introduce
unfamiliar notation as needed.  We do move between
scheme theory and Weil-style algebraic geometry.   For the most part, the theory of prolongation and jet spaces
must be developed scheme theoretically as we make essential use of nonreduced bases.  However, in the applications
to the first order theories of fields, as is common in model theoretic algebra, we sometimes use Weil-style language.
Let us note some of these conventions.  If $X$ is some variety over a field $K$, $L$ is an extension field of $K$, and
$a \in X(L)$ is an $L$-rational point, then $\operatorname{loc}(a/K)$, the locus of $a$ over $K$, is the intersection of
all closed $K$-subvarieties $Y \subseteq X$ with $a \in Y(L)$.   If $X$ is affine with coordinate ring ${\mathcal O}_X$, then
we define $I(a/K) := \{ f \in {\mathcal O}_X ~:~ f(a) = 0 \}$ to be the ideal of $a$ over $K$.   Generalizing somewhat,
for $Y \subseteq X_L$ a subvariety of the base change of $X$ to $L$, we define $I(Y/K) := \{ f \in {\mathcal O}_X ~:~ f
\text{ vanishes on } Y \}$.   Note that the locus of $a$ over
$K$ is the variety defined by the ideal of $a$ over $K$.  We say that $a$ is a generic point of $X$ over $K$ (or is
$K$-generic in $X$) if $\operatorname{loc}(a/K) = X$.    Scheme theoretically, one would say that $I(a/K)$ is the
generic point of $X$, but these two points of view will not appear in the same section.

\bigskip
\section{$\cD$-rings}
\label{Dringsec}

\noindent
Throughout this paper we will fix the following data, sometimes making further assumptions about them:
\begin{itemize}
\label{data}
\item[A.]
a base ring $A$
\item[B.]
a finite free $A$-algebra $\cD(A)$; that is, $\cD(A)$ is an $A$-algebra which as an $A$-module is free of finite rank,
\item[C.]
an $A$-algebra homomorphism $\pi^A:\cD(A)\to A$, and
\item[D.]
an $A$-basis $(\epsilon_0,\dots,\epsilon_{\ell-1})$ for $\cD(A)$ such that $\pi^A(\epsilon_0)=1$ and $\pi^A(\epsilon_i)=0$ for all $i=1,\dots,\ell-1$.
\end{itemize}
An equivalent scheme-theoretic way to describe these data is as a {\em finite free $\mathbb S$-algebra scheme over $A$ with basis} in the sense of~\cite{paperA} and~\cite{paperB}.
Here $\mathbb S$ denotes the ring scheme which when evaluated at any $A$-algebra $R$ is just the ring $R$ itself.
That is, $\mathbb S$ is simply the affine line $\spec\big(A[x]\big)$ endowed with the usual ring scheme structure.
Instead of~B, C, D as above we could consider the basic data as being
\begin{itemize}
\item[B$'$.]
an $\mathbb S$-algebra scheme $\cD$ over $A$; that is, a ring scheme $\cD$ over $A$ together with a ring scheme morphism $s:\mathbb S\to\cD$ over $A$,
\item[C$'$.]
a morphism of $\mathbb S$-algebra schemes $\pi:\cD\to\mathbb S$ over $A$, and
\item[D$'$.]
an $\mathbb S$-linear isomorphism $\psi:\cD\to\mathbb S^\ell$ over $A$ such that $\pi$ is $\psi$ composed with the first co-ordinate projection on $\mathbb S^\ell$.
\end{itemize}
Indeed, as is explained on page 14 of~\cite{paperB}, we obtain the second presentation from the first as follows: given any $A$-algebra $R$ define $\cD(R)=R\otimes_A\cD(A)$, $\pi^R=\id_R\otimes\pi^A$, and $\psi^R=\id_R\otimes\psi^A$ where $\psi^A:\cD(A)\to A^\ell$ is the $A$-linear isomorphism induced by the choice of basis $(\epsilon_0,\dots,\epsilon_{\ell-1})$.
To go in the other direction is clear, one just evaluates all the scheme-theoretic data on the ring $A$.
A key point is that given $(\cD,\pi,\psi)$, for any $A$-algebra $R$ there is a canonical isomorphism induced by $\psi$ between $\cD(R)$ and $R\otimes_A\cD(A)$.
While the first presentation of the data is more immediately accessible, it is the second scheme-theoretic one that is more efficient and that we will use.

\begin{remark}
\label{assumptionpi}
The assumption in~D$'$ that $\pi$ is $\psi$ composed with the first co-ordinate projection on $\mathbb S^\ell$ is new in that it was not made in Definition~2.2 of~\cite{paperB}.
However, it can always be made to hold through a change of basis.
\end{remark}

The multiplicative structure on $\cD(A)$, and hence on $\cD(R)$ for any $A$-algebra $R$, can be described in terms of the basis by writing
\begin{eqnarray}
\epsilon_i\epsilon_j&=&\label{productbasis}
\sum_{k=0}^{\ell-1}a_{i,j,k}\epsilon_k
\\
1_{\cD(A)}&=& \sum_{k=0}^{\ell-1}c_k\epsilon_k \label{1basis}
\end{eqnarray}
where the $a_{i,j,k}$'s and $c_k$'s are elements of $A$.  Note that $a_{0,0,0} = 1$ and $c_0 = 1$.

\begin{definition}[$\cD$-rings]
By a {\em $\cD$-ring} we will mean an $A$-algebra $R$ together with a sequence of operators $\partial:=(\partial_1,\dots,\partial_{\ell-1})$ on $R$ such that the map $e:R\to\cD(R)$ given by
$$e(r):=r\epsilon_0+\partial_1(r)\epsilon_1+\cdots+\partial_{\ell-1}(r)\epsilon_{\ell-1}$$
is an $A$-algebra homomorphism.
\end{definition}

Via the above identity we can move back and forth between thinking of a $\cD$-ring as $(R,\partial)$ or as $(R,e)$.
It should be remarked that this is not exactly consistent with~\cite{paperA}.
In that paper, a ``$\cD$-ring'' was defined to be simply an $A$-algebra $R$ together with an $A$-algebra homomorphism $e:R\to\cD(R)$.
Hence, under the correspondence $(R,\partial)\mapsto (R,e)$, the $\cD$-rings of the current paper are precisely the ``$\cD$-rings'' of~\cite{paperA} with the additional assumption that $e$ is a section to $\pi^R:\cD(R)\to R$.
(Note that this latter assumption already appears in Definition~2.4 of~\cite{paperB}.)

The class of $\cD$-rings is axiomatisable in the language
$$\cL_\cD := \{0,1,+,-,\times,(\lambda_a)_{a\in A},\partial_1,\dots,\partial_{\ell-1}\}$$\label{language}
where $\lambda_a$ is scalar multiplication by $a\in A$.
Indeed, the class of $A$-algebras is cleary axiomatisable, and the $A$-linearity of $e:R\to\cD(R)$, which is equivalent to $\partial_1,\dots,\partial_{\ell-1}$ being $A$-linear operators on $R$, is also axiomatisable.
Finally, that $e$ is in addition a ring homomorphism corresponds to the satisfaction of certain $A$-linear functional equations on the operators.
Indeed, using~(\ref{productbasis}) above, we see that the multiplicativity of $e$ is equivalent to
\begin{eqnarray}
\partial_k(xy) &=&
\sum_{i=0}^{\ell-1}\sum_{j=0}^{\ell-1}a_{i,j,k}\partial_i(x)\partial_j(y) \ \ \ \text{ for all $x,y$, and} \label{multrule}\\
\partial_k(1_R) &=& c_k \label{idrule}
\end{eqnarray}
for all $k=1,\dots,\ell-1$.

\begin{example}[Prime $\cD$-ring]
\label{primeexample}
For any $(A,\cD,\pi,\psi)$ there is a unique $\cD$-ring structure on $A$, namely where the $\partial_i:A\to A$ are $A$-linear and satisfy
$$1=\epsilon_0+\sum_{i=1}^{\ell-1}\partial_i(1)\epsilon_i$$
This corresponds to $e=s^A$, and is called the {\em prime $\cD$-ring}.
\end{example}

\begin{example}[Fibred products]
\label{combine}
We can always combine examples.
Given $(\cD,\pi,\psi)$ and $(\cD',\pi',\psi')$ we can consider the fibred product $\cD\times_{\mathbb S}\cD'$ with $\pi\times\pi':\cD\times_{\mathbb S}\cD'\to\mathbb S\times_{\mathbb S}\mathbb S=\mathbb S$, and $\psi\times\psi':\cD\times_{\mathbb S}\cD'\to\mathbb S^{\ell}\times_{\mathbb S}\mathbb S^{\ell'}=\mathbb S^{\ell+\ell'-1}$.
The $\cD\times_{\mathbb S}\cD'$-rings will be precisely those of the form $(R,\partial,\partial')$ where $(R,\partial)$ is a $\cD$-ring and $(R,\partial')$ is a $\cD'$-ring.
Note that the theory of $\cD\times_{\mathbb S}\cD'$-rings does not impose any nontrivial functional equation relating $\partial$ and $\partial'$, for example, they are not asked to commute.
\end{example}

\begin{example}[Tensor products]
\label{tensorexample}
Here is another way to combine examples.
Given $(\cD,\pi,\psi)$ and $(\cD',\pi',\psi')$ we can consider the tensor product $\cD\otimes_{\mathbb S}\cD'$ with $\pi\otimes\pi':\cD\otimes_{\mathbb S}\cD'\to\mathbb S$,
and $\psi\otimes\psi':\cD\otimes_{\mathbb S}\cD'\mathbb \to {\mathbb S}^{\ell}\otimes_{\mathbb S}\mathbb S^{\ell'}=\mathbb S^{\ell\ell'}$.
If $R$ has both a $\cD$-ring structure $(R,\partial)$ and a $\cD'$-ring
structure $(R,\partial')$, then it has the natural $(\cD \otimes_{\mathbb S} \cD')$-structure $(R,D)$ where $D_{i + j \ell} = \partial_j' \circ \partial_i$.
This comes from the fact that
\begin{eqnarray*}
(\cD \otimes_{\mathbb S} \cD')(R)&=&R\otimes_A\big(\cD(A) \otimes_A \cD'(A)\big)\\
&=&
\big(R\otimes_A\cD(A)\big) \otimes_A \cD'(A)\\
&=&\cD(R)\otimes_A \cD'(A)\\
&=&\cD'\big(\cD(R)\big)
\end{eqnarray*}
But not every $(\cD \otimes_{\mathbb S} \cD')$-structure on $R$ is of this form.
For example, one also has $(R,\widetilde{D})$ with $\widetilde{D}_{i + j \ell} = \partial_i\circ\partial_j'$ that arises from regarding $(\cD \otimes_{\mathbb S} \cD')$ as $\cD\circ\cD'$ instead.
The identities imposed by being
a $(\cD \otimes_{\mathbb S} \cD')$-ring are just the generalised Leibniz rules satisfied by compositions of the components of $\cD$-ring and
$\cD'$-ring structures.
\end{example}

\begin{example}
\label{examples}
We list here some of the main motivating examples.
Characteristic~$0$ or~$p$ specialisations of these examples are obtained by letting $A$ be $\mathbb Q$ or $\mathbb F_p$, respectively.
In what follows $R$ ranges over all $A$-algebras.
\begin{itemize}
\item[(a)]
{\em Differential rings.}
Let
$\cD(R)=R[\eta]/(\eta^2)$ with the natural $R$-algebra structure,
$\pi^R:R[\eta]/(\eta^2)\to R$ be the quotient map, and $(1,\eta)$ the $R$-basis.
Then a $\cD$-ring is precisely an $A$-algebra equipped with a derivation over $A$.
\item[(b)]
{\em Truncated higher derivations.}
Generalising the above example, let
$$\cD(R)=R[\eta]/(\eta^{n+1})$$
with the natural $R$-algebra structure,
$\pi^R:R[\eta]/(\eta^{n+1})\to R$ the quotient map, and take as an $R$-basis $(1,\eta,\dots,\eta^n)$.
Then a $\cD$-ring is precisely an $A$-algebra equipped with a {\em higher derivation of length $n$ over $A$} in the sense of~\cite{matsumura}; that is, a sequence of $A$-linear maps $(\partial_0=\id,\partial_1,\dots,\partial_m)$ such that $\partial_i(xy)=\sum_{r+s=i}\partial_r(x)\partial_s(y)$.

It is worth pointing out here that even in characteristic zero (so when $A=\mathbb Q$ for example) this is a proper generalisation of differential rings.
It is true that $\partial_1$ is a derivation, and {\em if} we had imposed the usual iterativity condition then we would have $\displaystyle \partial_i=\frac{\partial_1^i}{i!}$.
But the point is that being a $\cD$-ring does not impose iterativity, the operators are in this sense ``free".
\item[(c)]
{\em Difference rings.}
Let $\cD(R)=R^{2}$ with the product $R$-algebra structure,
$\pi^R$ the projection onto the first co-ordinate, and $(\epsilon_0,\epsilon_1)$ the standard basis.
Then a $\cD$-ring is precisely an $A$-algebra equipped with an $A$-endomorphism.
\item[(d)]
{\em Partial higher differential-difference rings.}
Taking fibred products as in Example~\ref{combine}, we can combine the above examples.
That is, suppose we are given positive integers $m_1, n_1,\dots,n_{m_1}$, and $m_2$.
For an appropriate choice of $(\cD,\pi,\psi)$ the $\cD$-rings will be precisely the $A$-algebras equipped with $m_1$ higher $A$-derivations (of length $n_1,\dots,n_{m_1}$ respectively) and $m_2$ $A$-endomorphisms.
Note that being a $\cD$-ring will not impose that the various operations commute.
\item[(e)]
{\em $D$-rings.}
Fix $c\in A$ and let $\cD(R) := R^2$ as an $R$-module and define multiplication by
$$(x_1,y_1) \cdot (x_2,y_2) := (x_1 x_2, x_1 y_2 + y_1 x_2 + y_1 y_2 c).$$
A
$\cD$-ring is then an $A$-algebra $R$ equipped with an $A$-linear map $D:R \to R$ satisfying the
twisted Leibniz rule $D(xy) = x D(y) + D(x) y + D(x) D(y) c$.
If we define $\sigma:R \to R$ by
$\sigma(x) := x + D(x)c$, then $\sigma$ is a ring endomorphism of $R$.
If $c = 0$ then $D$ is
a derivation on $R$, if $c$ is invertible in $A$ then $D$ may be computed from $\sigma$ by the
rule $D(x) = e^{-1}(\sigma(x) - x)$.
Such structures were considered by the second author in~\cite{scanlon2000}.
\end{itemize}
\end{example}

\begin{example}
In order to exhibit the variety of operators that can be put into this formalism, let us describe two more examples.
\begin{itemize}
\item[(a)] \emph{Derivations of an endomorphism.}
Consider $\cD(R)=R\times R[\eta]/(\eta^2)$ with basis $\{(1,0),(0,1),(1,\eta)\}$ and $\pi$ the projection onto the first co-ordinate.
A $\cD$-ring is then an $A$-algebra equipped with an $A$-endomorphism $\sigma$ and an $A$-linear map $\delta$ satisfying the $\sigma$-twisted Leibniz rule
$$\delta(xy)=\sigma(x)\delta(y)+\delta(x)\sigma(y)$$
\item[(b)]
Suppose $\mathcal{D}(R)=R[\eta_1,\eta_2]/(\eta_1^2,\eta_2^2)$, with basis  $\{1, \eta_1,\eta_2,\eta_1\eta_2\}$ and $\pi$ the natural quotient map.
A $\mathcal{D}$-ring is an $A$-algebra $R$ equipped with three operators, $\partial_1,\partial_2,\partial_3$, such that  $\partial_1$ and $\partial_2$ are derivations and $\partial_3$ is an $A$-linear map satisfying
$$\partial_3(xy)=x\partial_3(y)+y\partial_3(x)+\partial_1(x)\partial_2(y)+\partial_2(x)\partial_1(y)$$
For example, if $\partial_1,\partial_2$ are {\em arbitrary} derivations on $R$ and $\partial_3:=\partial_1\circ\partial_2$, then $(R,\partial_1,\partial_2,\partial_3)$ is a $\mathcal{D}$-ring.
This example is a special case of the tensor product construction of Example~\ref{tensorexample}.
\end{itemize}
\end{example}

This formalism of $\cD$-rings is rather general.
We leave to future work the systematical classification of the operators on  $A$-algebras which it covers.

\begin{remark}
\label{freeness}
It may at this point be worth explaining in what sense we consider the theory we develop here as a theory of ``free" operators.
To say that $(R,\partial_1,\dots,\partial_{\ell-1})$ is a $\cD$-ring certainly imposes some relations among the $\partial_i$, namely the multiplicativity of the corresponding $e:R\to\cD(R)$ forces the constraint~(\ref{multrule}) on page~\pageref{multrule}.
For example, in the case of a higher derivation this says that
$$\partial_i(xy)=\sum_{r+s=i}\partial_r(x)\partial_s(y)$$
But the point is that such constraints entailed by the multiplicativity of $e$ are the only non-trivial relation among the $\partial_i$ that are forced.
Of course, in a particular $\cD$-ring further functional equations may happen to hold -- the $\partial_i$ may commute, for example -- but this is not imposed by the theory of $\cD$-rings.
\end{remark}

In the next chapter we will show that  the theory of $\cD$-fields in characteristic zero has a model companion.
In order to describe the axioms of this model companion we will make use of the {\em abstract prolongations} introduced and discussed in~$\S 4$ of~\cite{paperA}.
Let us briefly recall them here.

Given a $\cD$-ring $(R,\partial)$ and an algebraic scheme $X$ over $R$, the prolongation\footnote{The prolongation does not always exist; however it does exist for any quasi-projective scheme. See the discussion after Definition~4.1 of~\cite{paperA} for details.}
of $X$, denoted by $\tau(X,\cD, e)$ or just $\tau X$ for short, is itself a scheme over $R$ with the characteristic property that its $R$-points can be canonically identified with\label{identifyprolong} $X(\cD(R))$ where
 $X$ is regarded as a scheme over $\cD(R)$ via the base change coming from $e:R \to \cD(R)$.
 Via this identification, note that $e$  induces a map $\nabla:X(R)\to\tau(X,\cD,e)(R)$.

 In terms of equations, if $X\subset\mathbb A_R^n$ is the affine scheme $\spec\big(R[x]/I\big)$ where $x = (x_1,\ldots,x_n)$ is really
 an $n$-tuple of indeterminates then $\tau X$ will be the closed subscheme of $\mathbb A_R^{n\ell}$ given by $\spec\big(R[x^{(0)},x^{(1)},\ldots,x^{(\ell-1)}]/I'\big)$ where $I'$ is obtained as follows:
For each $P(x)\in I$ let $P^e(x)\in\cD(R)[x]$ be the polynomial obtained by applying $e$ to the coefficients of $P$, and compute
\begin{eqnarray*}
P^e( \sum_{j=0}^{\ell - 1} x^{(j)} \epsilon_j)&=&
\sum_{j=0}^{\ell -1} P^{(j)}(x^{(0)},x^{(1)},\dots,x^{(\ell-1)})\epsilon_j
\end{eqnarray*}
in the polynomial ring $\displaystyle \cD(R)[x^{(0)},x^{(1)},\dots,x^{(\ell-1)}]=\bigoplus_{i=0}^{\ell-1}R[x^{(0)},x^{(1)},\dots,x^{(\ell-1)}]\cdot \epsilon_i$.
Then $I'$ is the ideal of $R[x^{(0)},x^{(1)},\dots,x^{(\ell-1)}]$ generated by $P^{(0)},\dots,P^{(\ell-1)}$ as $P$ ranges in $I$.
Note that the $P^{(i)}$'s are computed using~(\ref{productbasis}) above.
With respect to these co-ordinates, the map $\nabla:X(R)\to\tau(X,\cD,e)(R)$ is given by
$\nabla(a)=\big(a,\partial_1(a),\dots,\partial_{\ell-1}(a)\big)$.

\bigskip
\section{Existentially closed $\cD$-fields}
\label{chapter-mc}

\noindent
We aim to show that the theory of $\cD$-fields of characteristic zero admits a model companion when the base ring $A$ enjoys some additional properties that are spelled out with the following assumptions.
In fact, as is explained in the appendix, these assumptions are not necessary.
Nevertheless, for the sake of significant ease of notation, and in order to better fix ideas, we will impose the following:

\begin{assumption}
\label{assumptionA}
The following assumptions will be in place throughout the rest of the paper, unless explicitly stated otherwise:
\begin{itemize}
\item[(i)]
The ring $A$ is a field.
\item[(ii)]
Writing $\cD(A)=\prod_{i=0}^tB_i$, where the $B_i$ are  local finite $A$-algebras,
the residue field of each $B_i$, which is necessarily a finite extension of $A$, is in fact $A$ itself.
\end{itemize}
\end{assumption}
\noindent
Note that  the $B_i$'s are unique up to isomorphism and reordering of the indices.
They can be obtained
 by running through the finitely many maximal ideals of $\cD(A)$ and quotienting out by a sufficiently high power of them.

All of the motivating examples described in~\ref{examples}, specialised to the case when $A=\mathbb Q$ or $\mathbb F_p$, satisfy Assumptions~\ref{assumptionA}.
The following is an example where~\ref{assumptionA}(ii) is not satisfied.

\begin{example}
\label{no4ii}
Let $A=\mathbb Q$ and $\cD(\mathbb Q)=\mathbb Q \ \times \ \mathbb Q[x]/(x^2-2)$, with standard basis $\{(1,0),(0,1),(0,x)\}$.
Then, in the notation of~\ref{assumptionA}(ii) we have $t=1$, $B_0=\mathbb Q$, $B_1=Q[x]/(x^2-2)$, and the residue field of $B_1$ is $B_1$ itself.
So~\ref{assumptionA}(ii) is not satisfied.
The $\cD$-rings in this case are precisely the $\mathbb Q$-algebras $R$ equipped with linear operators $\partial_1,\partial_2$ such that
\begin{itemize}
\item
$\partial_1(ab)=\partial_1(a)\partial_1(b)+2\partial_2(a)\partial_2(b)$
\item
$\partial_2(ab)=\partial_1(a)\partial_2(b)+\partial_2(a)\partial_1(b)$
 \end{itemize}
 Note that if $(K,\partial_1,\partial_2)$ is a $\cD$-field with $\sqrt{2}\in K$, then $\partial_1,\partial_2$ are interdefinable with the pair of endomorphisms $\partial_1+\sqrt{2}\partial_2$ and $\partial_1-\sqrt{2}\partial_2$ of $K$.
 This gives a hint as to how we should handle the situation when~\ref{assumptionA}(ii) fails, see~$\S$\ref{subsect-assumptionA} below.
\end{example}

\medskip
\subsection{The associated endomorphisms}
\label{assocend-section}
Our axioms for the model companion must take into account certain definable endomorphisms that are induced by the $\cD$-operators given the above decomposition of $\cD(A)$ into local artinian $A$-algebras.

First some notation.
Fixing $A$-bases for $B_0,\dots, B_t$, we get
\begin{itemize}
\item
finite free local $\bbs$-algebra schemes with bases, $(\cD_i,\psi_i)$ for $i=0,\dots, t$, such that $\cD_i(A)=B_i$,
\item
$\bbs$-algebra homomorphisms $\theta_i:\cD\to\cD_i$ corresponding to $\cD=\prod_{i=0}^t\cD_i$,
 \item
 $\bbs$-algebra homomorphisms $\rho_i:\cD_i\to\bbs$ which when evaluated at $A$ are the residue maps $B_i\to A$, and
 \item
 $\pi_i:=\rho_i\circ\theta_i:\cD\to\bbs$.
 \end{itemize}
Note that one of the maximal ideals of $\cD(A)$, say the one corresponding to $B_0$, is the kernel of our $A$-algebra homomorphism $\pi^A:\cD(A)\to A$.
In particular, $\pi_0=\pi$.

Now suppose $(R,\partial)$ is a $\cD$-ring.
For each $i=0,\dots, t$, we have the $A$-algebra endomorphism $\sigma_i:=\pi_i^R\circ e:R\to R$.
Since $\pi_0=\pi$, $\sigma_0=\id$.
The others, $\sigma_1,\dots,\sigma_t$, will be possibly nontrivial endomorphisms of $R$.
As the $\pi_i$ are $\bbs$-linear morphisms over $A$, the $\sigma_i$ are $A$-linear combinations of the operators $\partial_1,\dots,\partial_{\ell-1}$.
In particular, these are $0$-definable in $(R,\partial)$.
We call them the {\em associated endomorphisms} and $(R,\sigma_1,\dots,\sigma_t)$ the {\em associated difference ring}.

\begin{example}
In the partial difference-differential case of Example~\ref{examples}(d), that is, of $(R,\partial,\sigma)$ where $\partial$ is a tuple of $m_1$ (higher truncated) derivations on $R$ and $\sigma$ is a tuple of $m_2$ endomorphisms of $R$,  the associated difference ring is, as expected, $(R,\sigma)$.
In the $D$-rings of Example~\ref{examples}(e),  the associated endomorphism is the map $\sigma(x) := x + D(x) c$.
\end{example}

\begin{definition}
An {\em inversive} $\cD$-ring is one for which the associated endomorphisms are surjective.
That is, a $\cD$-ring $(R,\partial)$ is inversive just in case the associated difference ring $(R,\sigma)$ is inversive.

As a consequence of Theorem~\ref{ecedomains} below, we will see that if $R$ is an integral domain of characteristic zero whose associated endomorphisms are injective, then $(R,\partial)$ embeds into an inversive $\cD$-field.

Suppose now that $(K,\partial)$ is an inversive $\cD$-field.
The associated endomorphisms are thus automorphisms of $K$.
Given $B\subseteq K$,
the {\em inversive closure of $B$ in $K$}, denoted by $\langle B\rangle$, is the smallest inversive $\cD$-subring of $K$ containing $B$.
That is, it is the intersection of all inversive $\cD$-subrings containing $B$.
If $R$ is an inversive $\cD$-subring of $K$ and $a=(a_1,\dots,a_n)$ is a tuple from $K$, then $R\langle a\rangle$ is used to abbreviate $\big\langle R\cup\{a_1,\dots,a_n\}\big\rangle$.
\end{definition}

\begin{remark}
\begin{itemize}
\item[(a)]
Because $\sigma$ and $\partial$ may not commute, it is not the case that $\langle B\rangle$ is simply the inversive difference ring generated by the $\cD$-subring generated by $B$.
Rather, $\langle B\rangle=\bigcup_{i<\omega}R_i$ where $R_{-1}=B$ and $R_i$ is the $\cD$-subring generated by $\{\sigma_j^{-1}(a):a\in R_{i-1}, j=1,\dots,t\}$.
\item[(b)]
We denote by $\Theta a$ the (infinite) tuple whose co-ordinates are of the form $\theta a_i$ as $\theta$ ranges over all finite words on the set $\{\partial_1,\dots,\partial_{\ell-1},\sigma_1^{-1},\dots,\sigma_t^{-1}\}$.
So $k\langle a\rangle=k[\Theta a]$.
\end{itemize}
\end{remark}

\medskip
\subsection{The model companion}
\label{mc}
Our model companion for $\cD$-fields will include a ``geometric''  axiom in the spirit of Chatzidakis-Hrushovski~\cite{acfa1} or Pierce-Pillay~\cite{piercepillay}.
To state these we need some further notation.
Suppose $X$ is an algebraic scheme over a $\cD$-ring $R$.
For each $i=0,\dots,t$, since $\sigma_i=\pi_i^R\circ e$, the morphism $\pi_i:\cD\to\bbs$ induces\footnote{See~$\S4.1$ of~\cite{paperA} for the construction.} a surjective morphism on the prolongations  $\widehat\pi_i:\tau(X,\cD,e)\to \tau(X,\mathbb S,\sigma_i)$.
For ease of notation we will set $X^{\sigma_i}:=\tau(X,\mathbb S,\sigma_i)$.
Note that $X^{\sigma_i}$ is nothing other than $X$ base changed via $\sigma_i:R\to R$,  and that in terms of equations it is obtained by applying $\sigma_i$ to the coefficients of the defining polynomials.

\begin{theorem}
\label{ecedomains}
With Assumptions~\ref{assumptionA} in place, let $\cK$ denote the class of $\cD$-rings $(R,\partial)$ such that $R$ is an integral domain of characteristic zero and the associated endomorphisms are injective.
Then $(K,\partial)\in\cK$ is existentially closed if and only if
\begin{itemize}
\item[I.]
$K$ is an algebraically closed field,
\item[II.]
$(K,\partial)$ is inversive, and
\item[III.]
if $X$ is an irreducible affine variety over $K$ and $Y\subseteq\tau(X,\cD,e)$ is an irreducible subvariety over $K$ such that $\widehat{\pi_i}(Y)$ is Zariski dense in $X^{\sigma_i}$ for all $i=0,\dots,t$, then there exists $a\in X(K)$ with $\nabla(a)\in Y(K)$.
\end{itemize}
\end{theorem}

Note that $\cK$ is a universally axiomatisable class.
(This uses the fact that the associated endomorphisms of a $\cD$-ring are $A$-linear combinations of the operators, and hence definable.)
Moreover, the characterisation of existentially closed models given in the theorem is also first-order.
Indeed, the only one that is not obviously elementary is condition~III, but
since irreducibility and Zariski-density are parametrically definable in algebraically closed fields, the only thing to check is that if $X$ varies in an algebraic family then so do the $\widehat{\pi_i}:\tau X\to X^{\sigma_i}$.
That $\tau X$ and $X^{\sigma_i}$ vary uniformly in families follows from Proposition~4.7(b) of~\cite{paperA}, but can also be verified directly by looking at the equations that define the prolongations.
That $\widehat{\pi_i}$ also varies algebraically follows from the construction of these morphisms in $\S 4.1$ of~\cite{paperA}; see in particular Proposition~4.8(b) of that paper.
The following is therefore an immediate corollary of Theorem~\ref{ecedomains}:

\begin{corollary}
\label{corecedomains}
Under assumptions~\ref{assumptionA} the theory of $\cD$-fields of characteristic zero admits a model companion.
We denote the model companion by $\ecdf$, and we call its models {\em $\cD$-closed fields}.
\end{corollary}

We now work toward a proof of Theorem~\ref{ecedomains}.

To prove properties~I through~III of Theorem~\ref{ecedomains} for every existentially closed model in $\cK$ is to prove various extension lemmas about $\cK$.
In order to facilitate this we introduce the following auxiliary class.

\begin{definition}[The class $\cM$]
The class $\cM$ is defined to be the class of triples $(R,S,\partial)$ where $R\subseteq S$ are integral $A$-algebras of characteristic zero and $\partial=(\partial_1,\dots,\partial_{\ell-1})$ is a sequence of maps from $R$ to $S$ such that $e:R\to\cD(S)$ given by $e(r):=r\epsilon_0+\partial_1(r)\epsilon_1+\cdots+\partial_{\ell-1}(r)\epsilon_{\ell-1}$ has the following properties:
\begin{itemize}
\item[(i)]
$e$ is an $A$-algebra homomorphism,
\item[(ii)]
 for each $i=1,\dots, t$, $\sigma_i:=\pi_i^S\circ e:R\to S$ is injective.
\end{itemize}
Note that $\sigma_0:=\pi_0^S\circ e=\pi^S\circ e$ is then the inclusion map.
\end{definition}

So $(R,\partial)\in\cK$ if and only if $(R,R,\partial)\in\cM$.

The following lemma will imply that every existentially closed member of $\mathcal K$ is a field.
It is here that we require the associated endomorphisms to be injective.

\begin{lemma}
\label{extendtofields}
Suppose $(R,L,\partial)\in\cM$ with $L$ a field.
Then we can (uniquely) extend $\partial$ to the fraction field $F$ of $R$ so that $(F,L,\partial)\in\cM$.
\end{lemma}

\begin{proof}
Let us first extend $e:R\to\cD(L)$ to an $A$-algebra homomorphism from $F$ to $\cD(L)$.
By the universal property of localisation it suffices (and is necessary) to show that  $e$ takes nonzero elements of $R$ to units in $\cD(L)$.
Note that an element $x\in\cD(L)$ is a unit if and only if each of its projections $\theta^L_i(x)$ is a unit in the local $L$-algebra $\cD_i(L)$, which in turn is equivalent to the residue of $\theta^L_i(x)$, namely $\pi^L_i(x)\in L$, being nonzero.
But for all $i=1,\dots,t$, $\pi^L_i\circ e$ is injective on $R$ by assumption.
So $e(a)$ is a unit in $\cD(L)$ for nonzero $a\in R$.

We thus have an extension $e:F\to\cD(L)$.
The injectivity of $\sigma_1,\dots,\sigma_t$ is immediate as these are $A$-algebra homomorphisms between fields.
Moreover, $\pi\circ e$ is the identity on $F$ as it is extends the identity on $R$.
So, letting $\partial$ be the operators corresponding to $e:F\to\cD(L)$, we have that $(F,L,\partial)\in\cM$.
\end{proof}

The next lemma shows that existentially closed models are algebraically closed.

\begin{lemma}
\label{extendtoalg}
Suppose $(F,L,\partial)\in\cM$ where $F$ and $L$ are fields and $L$ is algebraically closed.
Then we can extend~$\partial$ to $F^{\alg}$ so that $(F^{\alg},L,\partial)\in\cM$.

In fact, more is true.
If $\sigma$ is the tuple of embeddings $F\to L$ associated to $\partial$, and $\sigma'$ is any extension of $\sigma$ to $F^{\alg}$, then there is exactly one extension $\partial'$ of $\partial$ to $F^{\alg}$ with associated embeddings $\sigma'$.
\end{lemma}

\begin{proof}
By iteration it suffices to prove, for any given $a\in F^{\alg}$, that we can
extend~$e$ to $F(a)$ in such a way that $\pi^L\circ e$ is still the identity on $F(a)$.
Let $P(x)\in F[x]$ be the minimal polynomial of $a$ over $F$.
Fixing $i=1,\ldots,t$, we let $c_i\in L$ be a root for $P^{\sigma_i}(x)\in L[x]$, where $\sigma_i:F\to L$ is the field embedding $\pi_i^L\circ e:F\to L$.
We cover the $i=0$ case by letting $c_0:=a$; note that as $P^{\sigma_0}(x)=P(x)$, $c_0$ is a root of $P^{\sigma_0}(x)$.
Now, let $e_i:F\to \cD_i(L)$ be $\theta_i^L\circ e$.
Note that by construction $P^{\sigma_i}(x)\in L[x]$ is the reduction of $P^{e_i}(x)\in\cD_i(L)[x]$ modulo the maximal ideal of the local artinian ring $\cD_i(L)$.
Since $P^{\sigma_i}(x)$ is separable (we are in characteristic zero), Hensel's Lemma allows us to lift~$c_i$ to a root~$b_i$ of $P^{e_i}(x)$ in $\cD_i(L)$.
In fact there is a unique such lifting as $F \hookrightarrow F(a)$ is \'{e}tale.
Then $b=(b_0,\dots,b_t)\in\cD(L)$ is a root of $P^e(x)\in\cD(L)[x]$, and we can extend $e$ to $F(a)$ by sending $a$ to $b$.
By construction $\pi^Le(a)=\pi_0^Le(a)=c_0=a$, so that $\pi^L\circ e=\id_{F(a)}$.

In the above argument any choice of roots $c_1,\dots,c_t$ works, and once that choice is made there is a unique possibility for $b$.
This leads to the ``in fact'' clause of the lemma.
\end{proof}

The next lemma will imply that existentially closed models are inversive.

\begin{lemma}
\label{extendtoautos}
Suppose $(F,L,\partial)\in\cM$ where $F$ and $L$ are fields.
Then there exists an extension $L'$ of $L$ and an extension of $\partial$ to an inversive $\cD$-field structure on $L'$.
\end{lemma}

\begin{proof}
First, we can extend $\sigma_1,\dots,\sigma_t$ to automorphisms $\sigma_1',\dots,\sigma_t'$ of some algebraically closed $L'\supseteq L$.
Now fix a transcendence basis $B$ for $L'$ over $F$.
For each $b\in B$ let $b_0\in\cD_0(L')$ lift $b$, and let $b_i\in\cD_i(L')$ lift $\sigma_i'(b)$ for $i=1,\dots,t$.
Then define $e(b)$ to be $(b_0,\dots,b_t)\in\cD(L')$.
This gives us an extension $F[B]\to\cD\big(L'\big)$ of $e$ such that $\pi^{L'}\circ e,\ \pi_1^{L'}\circ e,\dots,\pi_t^{L'}\circ e$ agree with $\id,\sigma_1',\dots,\sigma_t'$, respectively.
By Lemma~\ref{extendtofields} we can extend $e$ to $F(B)\to\cD(L')$ preserving this property.
By Lemma~\ref{extendtoalg} we can extend $e$ to a $\cD$-structure on $L'=F(B)^{\alg}$ in such a way that the associated endomorphisms remain the automorphisms $\sigma_1',\dots,\sigma_t'$.
\end{proof}

\begin{proof}[Proof of Theorem~\ref{ecedomains}]
Suppose $(K,\partial)\in \cK$ is existentially closed.
By Lemmas~\ref{extendtofields}, \ref{extendtoalg}, and~\ref{extendtoautos} we know that $K$ is an algebraically closed field and that $\sigma_1,\dots,\sigma_t$ are automorphisms of $K$.
It remains to check condition~III.
Let $X\subseteq\mathbb A^n_K$ and $Y\subseteq\tau X$ be as in that condition.
Let $L$ be an algebraically closed field
extending $K$ and let $b \in Y(L)$ be
a $K$-generic point of $Y$.
Let $a:=\hat\pi(b)\in X({L})$.
Our goal is to extend $\partial$ to a $\cD$-field structure on some extension of $L$ in such a way that $\nabla(a)=b$.
This will suffice, because then by existential closedness there must exist a $K$-point of $X$ with the property that its image under $\nabla$ is a $K$-point of $Y$.

As described in~$\S4$ of~\cite{paperA}, $\tau X({L})$ can be canonically identified with the $\cD({L})$-points of the affine scheme over $\cD({K})$ obtained from $X$ by applying $e$ to the coefficients of the defining polynomials.
Let $b'$ be the $n$-tuple from $\cD({L})$ that corresponds to $b\in \tau X({L})$ under this identification.
So $P^e(b')=0$ for all $P(x)\in I(X/K)$.
Since $\hat\pi(Y)=\hat\pi_0(Y)$ is Zariski dense in $X^{\sigma_0}=X$ and $b$ is $K$-generic in $Y$, we have that $a=\hat\pi(b)$ is $K$-generic in $X$.
So $I(X)=I(a/K)$.
We thus have that $P^e(b')=0$ for all $P(x)\in I(a/K)$.
That is, we can extend $e:K\to\cD({L})$ to an $A$-algebra homomorphism $e:K[a]\to\cD({L})$ by $e(a)=b'$.
The fact that $\hat\pi(b)=a$ implies that $\pi^{L}(b')=a$ so that $\pi^{L}\circ e=\id_{K[a]}$.
For each $i=1, \ldots, t$, the fact that $\hat\pi_i(Y)$ is Zariski dense in $X^{\sigma_i}$ implies that $\hat\pi_i(b)$ is $K$-generic in $X^{\sigma_i}$, and hence for any $P(x)\in K[x]$ on which $a$ does not vanish, $\pi^{L}_ie\big(P(a)\big)=P^{\sigma_i}\big(\pi^{L}_i(b')\big)\neq 0$.
That is, $\pi_i^{L}\circ e:K[a]\to{L}$ is injective for each $i=1,\dots,t$.
So, letting $\partial$ be the corresponding operators, we have $\big(K[a],{L},\partial\big)\in\cM$.
By Lemma~\ref{extendtofields} this extends to $\big(K(a),{L},\partial\big)\in\cM$.
By Lemma~\ref{extendtoautos}, there is an extension ${L}'$ of ${L}$ such that $\partial$ extends to a $\cD$-field structure on ${L}'$.
The fact that $e(a)=b'$ implies that $\nabla(a)=b$, as desired.

Now for the converse.
Suppose $(K,\partial)$ is a $\cD$-field satisfying~I through~III.
To show that $(K,\partial)$ is existentially closed in $\cK$ it suffices to consider a conjunction of atomic
$\cL_{\cD,K}$-formulae that is realised in some extension of $(K,\partial)$ in $\cK$, and show that it is already realised in $(K,\partial)$.
Indeed, all inequations of the form $t(x_1,\dots,x_m)\neq 0$ that might appear can be replaced by $t(x_1,\dots,x_m)y-1=0$ where $y$ is a new variable.
We can also assume, by Lemmas~\ref{extendtofields} and~\ref{extendtoalg}, that the extension in which we have a realisation is an algebraically closed $\cD$-field.

Let $\phi(x)$ be a conjunction of atomic $\cL_{\cD,K}$-formulae where $x=(x_1,\dots,x_m)$ is an $m$-tuple of variables, let $(L,\partial)$ be an algebraically closed $\cD$-field extension of $(K,\partial)$, and let $c_0\in L^m$ realise $\phi(x)$.
Let $\Xi$ be the set of all finite words on $\{\partial_1,\dots,\partial_{\ell-1}\}$, and for each $r\geq 0$ let $\Xi_r$ be those words of length at most $r$.
Fix an enumeration of $\Xi$ so that  $\Xi_r$ is an initial segment of $\Xi_{r+1}$ for all $r\geq 0$.
Define $\nabla_r:L\to L^{n_r}$ by $b\mapsto\big(\xi(b):\xi\in\Xi_r\big)$, where $n_r:=|\Xi|$.
Then for some $r\geq 0$, $\phi(x)^L=\{b\in L^m:\nabla_r(b)\in Z\}$ where $Z\subseteq L^{mn_r}$ is a Zariski-closed set over $K$.
Note that if $r=0$ then $\phi(x)$ is equivalent to a formula over $K$ in the language of rings with a realisation in an extension, and so, as $K$ is algebraically closed, $\phi(x)$ is realised in $K$.
We may thus assume that $r>0$.
Let
\begin{eqnarray*}
c &:=& \nabla_{r-1}(c_0)\in L^{mn_{r-1}}\\
X &:=& \operatorname{loc}(c/K)\subseteq L^{mn_{r-1}}\\
Y &:=& \operatorname{loc}(\nabla c/K)\subseteq\tau X(L)\subseteq L^{\ell mn_{r-1}}
\end{eqnarray*}
Note that for each $i=0,\ldots, t$,  $\hat\pi_i(\nabla c)=\sigma_i(c)\in X^{\sigma_i}(L)$.
Since $\nabla c$ is $K$-generic in $Y$ and $\sigma_i(c)$ is $K$-generic in $X^{\sigma_i}$ (as $\sigma_i$ restricts to an automorphism of $K$ by assumption), it follows that $\hat\pi_i(Y)$ is Zariski dense in $X^{\sigma_i}$.
Hence, by~III, there exists $a\in X(K)$ such that $\nabla a\in Y(K)$.
Let $a_0$ be the first $m$ co-ordinates of $a$.
It remains to verify that $a_0$ satisfies $\phi(x)$; that is, that $\nabla_r(a_0)\in Z(K)$.

First of all, we note that $\nabla_{r-1}(a_0)=a$.
Indeed, we show by induction on the length of $\xi\in\Xi_{r-1}$ that $\xi(a_0)=a_{\xi}$, where $a=(a_{\xi}:\xi\in\Xi_{r-1})$.
For $\xi=\id$ this is clear by choice of $a_0$.
Now suppose $\xi=\partial_i\xi'$.
Since $\nabla_{r-1}(c_0)=c$, we know that $\partial_ic_{\xi'}=c_{\xi}$.
Because $\nabla a$ is in the $K$-locus of $\nabla c$, we have $\partial_ia_{\xi'}=a_{\xi}$ also.
But by the inductive hypothesis, $\partial_ia_{\xi'}=\partial_i \xi'(a_0)=\xi(a_0)$, so that $\xi(a_0)=a_{\xi}$ as desired.

Finally, since $c_0$ is a realisation of $\phi(x)$, we know that $\nabla_r c_0\in Z$.
The latter can be seen as an algebraic fact about $\nabla\nabla_{r-1}c_0$.
Since $\nabla\nabla_{r-1}a_0=\nabla a$ is in the $K$-locus of $\nabla\nabla_{r-1}c_0=\nabla c$, it follows that $\nabla_r a_0\in Z$, as desired.

This completes the proof of Theorem~\ref{ecedomains}.
\end{proof}

Theorem~\ref{ecedomains} specialised to the various examples, say in~\ref{examples}, will yield model companions for a variety of theories of fields with operators.
In the classical examples one recovers the known ``geometric'' axiomatisations.
We conclude by pointing out that the difference field associated to a $\cD$-closed field is {\em difference-closed}.
Note that, as a consequence of this, the theory of $\cD$-fields imposes no non-trivial functional equations on the associated endomorphisms.
%writing down the case of partial difference fields, only because this case does not seem to appear explicitly in the literature.\marginpar{\tiny Look in literature.}

%\begin{corollary}
%\label{acfa}
%The theory of fields of chatacteristic zero equipped with $t$ (not necessarilly commuting) automorphisms has a model companion, which we denote by $\operatorname{ACFA}_{0,t}$, and which is axiomatised by
%\begin{itemize}
%\item[I.]
%$K$ is an algebraically closed field,
%\item[II.]
%$\sigma_1,\dots,\sigma_t$ are automorphisms of $K$,
%\item[III.]
%if $X$ is an irreducible affine variety over $K$ and $Y\subseteq X\times X^{\sigma_1}\times\cdots\times X^{\sigma_t}$ is an irreducible subvariety over $K$ whose projections onto each factor are Zariski-dense, then there exists $a\in X(K)$ with
%$$\big(a,\sigma_1(a),\dots,\sigma_t(a)\big)\in Y(K)$$
%\end{itemize}
%\end{corollary}

%\begin{proof}
%Apply Theorem~\ref{ecedomains} in the case when $A=\mathbb Q$ and $\cD(\mathbb Q)=\mathbb Q^{t+1}$ with the product $\mathbb Q$-algebra structure and the standard $\mathbb Q$-basis.
%\end{proof}

\begin{proposition}
\label{assocacfa}
If $(K,\partial)\models\ecdf$ and $(K,\sigma)$ is the associated difference field, then $(K,\sigma)\models\operatorname{ACFA}_{0,t}$, that is, it is an existentially closed model of the theory of fields of chatacteristic zero equipped with $t$ (not necessarily commuting) automorphisms.
\end{proposition}

\begin{proof}
The axioms for $\operatorname{ACFA}_{0,t}$ appear in $\S1.2$ of~\cite{HrMM}.
They say that $K$ should be algebraically closed, $\sigma_1,\dots,\sigma_t$ should be automorphisms of $K$, and, the only one that requires checking, {\em if $X$ is an irreducible affine variety over $K$ and $Y\subseteq X\times X^{\sigma_1}\times\cdots\times X^{\sigma_t}$ is an irreducible subvariety over $K$ whose projections onto each factor are Zariski-dense, then there should exist $a\in X(K)$ with $\big(a,\sigma_1(a),\dots,\sigma_t(a)\big)\in Y(K)$}.
To check this, consider the pull-back $Y'$ of $Y$ under $\tau(X,\cD,e)\to X\times X^{\sigma_1}\times\cdots\times X^{\sigma_t}$, and apply axiom~III of Theorem~\ref{ecedomains} to an irreducible component of $Y'$ that projects dominantly onto $Y$ (there will be one).
\end{proof}

\bigskip
\section{Basic model theory of $\ecdf$}
\label{basicmodel}

\noindent
We begin now to investigate the model theory of $\ecdf$ using the study of existentially closed difference fields as it appears in~$\S 1$ of~\cite{acfa1} as a template.
Assumptions~\ref{assumptionA} and the notation of the previous chapter remain in place.

\subsection{Completions}
We aim to describe the completions of $\ecdf$.

\begin{lemma}
\label{amalgam}
Suppose $(K,\partial)$ and $(L,\gamma)$ are inversive $\cD$-fields extending an inversive $\cD$-field $(F,\partial)$ with $K$ and $L$ linearly disjoint over $F$ (inside some fixed common field extension).
Then we can simultaneously extend $(K,\partial)$ and $(L,\gamma)$ uniquely to a $\cD$-field structure on the compositum $KL$.
\end{lemma}

\begin{proof}
It follows from linear disjointedness that $R:=K\otimes_F L$ is an integral domain whose fraction field is the compositum $KL$.
Here we identify $K$ with $K\otimes 1\subset R$, $L$ with $1\otimes L\subset R$ and $F$ with $1\otimes F=F\otimes 1\subseteq K\cap L$.
It suffices to find a common extension of $\partial$ and $\gamma$ to $R$.
Indeed, by Lemma~\ref{extendtofields} we can then further extend $R$ uniquely to a $\cD$-field structure on the fraction field $KL$.

Let $e_1:K\to\cD(K)\subseteq\cD(R)$ and $e_2:L\to\cD(L)\subseteq\cD(R)$ be the corresponding $A$-algebra homomorphisms.
Since these agree on $F$ we have the induced map $e:R\to\cD(R)$ determined by $e(a\otimes b):=e_1(a)e_2(b)$, which is easily seen to be an $A$-algebra homomorphism that extends both $e_1$ and $e_2$.
For $i=0,\dots,t$,
\begin{eqnarray*}
\pi^R_ie(a\otimes b)
&=&\big(\pi^R_ie_1(a)\big)\big(\pi^R_ie_2(b))\\
&=&\big(\pi^{K}_ie_1(a)\big)\big(\pi^{L}_ie_2(b))\\
&=&(\sigma_i(a)\otimes 1\big)\big(1\otimes\tau_i(b)\big)\\
&=&\sigma_i(a)\otimes\tau_i(b)
\end{eqnarray*}
where the $\sigma_is$ and $\tau_i$'s are the associated automorphisms of $K$ and $L$ respectively.
Applying this to $i=0$ we see that $\pi^R\circ e=\pi_0^R\circ e$ is the identity on $R$; hence $(R,e)$ is a $\cD$-ring.
For $i\geq 1$, since $\sigma_i$ and $\tau_i$ extend an automorphism of $F$ (by the inversiveness assumption) and $K$ is linearly disjoint from $L$ over $F$, $a\otimes b\mapsto\sigma_i(a)\otimes\tau_i(b)$ determines an automorphism of $R=K\otimes_F L$.
Hence $\pi_i^R\circ e$ is an automorphism of $R$ for $i=1,\dots,t$.
So $(R,e)$ is in the class $\cK$,  and extends $(K,e_1)$ and $(L,e_2)$, as desired.
\end{proof}

\begin{proposition}
\label{overclosed}
If $(K,\partial)$ and $(L,\gamma)$ are models of $\ecdf$ with a common algebraically closed inversive $\cD$-subfield $F$, then $(K,\partial)\equiv_F(L,\gamma)$.
\end{proposition}

\begin{proof}
The fact that $F$ is algebraically closed allows us to assume, after possibly replacing $(L,\gamma)$ by an $F$-isomorphic copy, that as subfields of some common field extension, $K$ and $L$ are linearly disjoint over $F$.
By Lemma~\ref{amalgam} we can extend $(K,\partial)$ and $(L,\gamma)$ simultaneously to a $\cD$-field structure on $KL$, which we can then extend further to a model, say $(K',\partial)\models\ecdf$.
By model completeness, $(K,\partial)\preceq(K',\partial)$ and $(L,\gamma)\preceq(K',\partial)$.
It follows that $(K,\partial)\equiv_F(L,\gamma)$.
\end{proof}

\begin{lemma}
\label{algsubstructure}
Suppose $(F,\partial)\subseteq(K,\partial)$ is a $\cD$-field extension such that $K$ is algebraically closed and $F$ is inversive.
Then $F^{\alg} \subseteq K$ is an inversive $\cD$-subfield.
\end{lemma}

\begin{proof}
Inversiveness comes for free once we see that $F^{\alg}$ is a $\cD$-subfield.
Let $a\in F^{\alg}$ and $P(x)\in F[x]$ be the minimal poynomial of $a$.
Let $e:K\to\cD(K)$ be the $A$-algebra homomorphism corresponding to $\partial$.
We need to show that $e(a)\in\cD(F^{\alg})$.
Under the identification $\cD(K)=\prod_{i=0}^t\cD_i(K)$, we have $e(a)=\big(e_0(a),\dots,e_t(a)\big)$, where $e_i:=\theta_i^K\circ e$, and it suffices to
show that each $e_i(a)\in\cD_i(F^{\alg})$.
Now $\sigma_i(a)\in F^{\alg}$ and by the inversiveness assumption $P^{\sigma_i}(x)$ is the minimal polynomial of $\sigma_i(a)$ over $F$.
So by Hensel's Lemma $\sigma_i(a)$ has a lifting to a root of $P^{e_i}(x)$ in $\cD_i(F^{\alg})$.
On the other hand, $e_i(a)$ also lifts $\sigma_i(a)$ to a root of $P^{e_i}(x)$ in $\cD_i(K)$.
As the extension is \`{e}tale  these liftings agree, and so $e_i(a)\in\cD_i(F^{\alg})$, as desired.
\end{proof}

\begin{corollary}[Completions of $\ecdf$]
\label{completions}
The completions of $\ecdf$ are determined by the difference-field structure on the algebraic closure of the prime $\cD$-field.
That is,  two models $(K,\partial)$ and $(L,\gamma)$ of $\ecdf$ are elementarily equivalent if and only if $(A^{\alg},\sigma\upharpoonright_{A^{\alg}})\approx_A(A^{\alg},\tau\upharpoonright_{A^{\alg}})$, where $\sigma$ and $\tau$ are the sequences of automorphisms of $K$ and $L$ associated to $\partial$ and $\gamma$, respectively.
\end{corollary}

\begin{proof}
First of all, both $(K,\partial)$ and $(L,\gamma)$ extend the prime $\cD$-field $A$, which is itself inversive (the difference-field structure on $A$ is trivial).
Hence, by Lemma~\ref{algsubstructure}, $(A^{\alg},\partial\upharpoonright_{A^{\alg}})$ and $(A^{\alg},\gamma\upharpoonright_{A^{\alg}})$ are inversive $\cD$-field extensions of $A$.
By Lemma~\ref{extendtoalg} their $\cD$-field structures are determined by the action of the corresponding automorphisms on $A^{\alg}$.
Hence, if $(A^{\alg},\sigma\upharpoonright_{A^{\alg}}))$ and $(A^{\alg},\tau\upharpoonright_{A^{\alg}}))$ are isomorphic then $(A^{\alg},\partial\upharpoonright_{A^{\alg}})$ and $(A^{\alg},\gamma\upharpoonright_{A^{\alg}})$ are isomorphic, and so by Proposition~\ref{overclosed}, $(K,\partial)$ and $(L,\gamma)$ are elementarily equivalent.
For the converse, if $(K,\partial)\equiv (L,\gamma)$ then there is an elementary embedding of $(K,\partial)$ into an elementary extension $(L',\gamma)$ of $(L,\gamma)$.
This elementary embedding will restrict to an isomorphism from $(A^{\alg},\partial\upharpoonright_{A^{\alg}})$ to its image in $(L',\gamma)$, which is $(A^{\alg},\gamma\upharpoonright_{A^{\alg}})$.
In particular, $(A^{\alg},\sigma\upharpoonright_{A^{\alg}})\approx_A(A^{\alg},\tau\upharpoonright_{A^{\alg}})$.
\end{proof}

\medskip
\subsection{Algebraic closure}
We characterise model-theoretic algebraic closure.

\begin{proposition}
\label{acl}
Suppose $(K,\partial)\models\ecdf$.
For all $B\subseteq K$, $\acl(B)=\langle B\rangle^{\alg}$.
\end{proposition}

\begin{proof}
Recall that $\langle B\rangle$ is the inversive closure of $B$, the smallest inversive $\cD$-subring of $K$ containing $B$.
As $\sigma_1,\dots,\sigma_t$ are $\cL_\cD$-definable, $\langle B\rangle\subseteq\dcl(B)$.
Hence $F:=\langle B\rangle^{\alg}\subseteq\acl(B)$.
It remains to show that if $a\in K\setminus F$ then $\tp(a/F)$ is nonalgebraic.

Note that, by Lemma~\ref{algsubstructure}, $F$ is an inversive $\cD$-subfield of $K$.
Since $F$ is algebraically closed we can find, in some common field extension, an isomorphic copy of $K$ over $F$, witnessed say by an $F$-isomorphism $\alpha:K\to K'$, and such that $K$ is linearly disjoint from $K'$ over~$F$.
Via $\alpha$ we can put a $\cD$-field structure $\partial'$ on $K'$ that extends $(F,\partial)$ and so that $\alpha$ is an isomorphism of $\cD$-fields.
Now we extend $(K,\partial)$ and $(K',\partial')$ to a $\cD$-field structure on $KK'$ using Lemma~\ref{amalgam}, and then further to a model of $\ecdf$.
We have thus found a common elementary extension of $(K,\partial)$ and $(K',\partial')$.
In this elementary extension, $\alpha(a)$ will be a realisation of $\tp(a/F)$ that is distinct from~$a$.
Iterating, we find infinitely many realisations of $\tp(a/F)$ in some elementary extension, proving that this type is nonalgebraic.
\end{proof}

\medskip
\subsection{Types}
\label{subsection-types}
We characterise types and deduce a quantifier reduction theorem.
Recall that $\Theta a$  is the (infinite) tuple whose co-ordinates are of the form $\theta a_i$ as $\theta$ ranges over all finite words on the set $\{\partial_1,\dots,\partial_{\ell-1},\sigma_1^{-1},\dots,\sigma_t^{-1}\}$.

\begin{proposition}
\label{types}
Suppose $(K,\partial)\models\ecdf$, $k\subseteq K$ is an inversive $\cD$-subfield, and $a,b\in K^n$.
Then the following are equivalent:
\begin{itemize}
\item[(i)]
$\tp(a/k)=\tp(b/k)$,
\item[(ii)]
$\tp_{\sigma}\big(\Theta a/k\big)=\tp_{\sigma}\big(\Theta b/k\big)$ (where $\tp_{\sigma}(c/k)$ denotes the
type of $c$ over $k$ in the reduct to the language of difference fields),
\item[(iii)]
there is an isomorphism from $\big(k\langle a\rangle,\partial\big)$ to $\big(k\langle b\rangle,\partial\big)$ sending $a$ to $b$ and fixing~$k$ that extends to an isomorphism from $\big(k\langle a\rangle^{\alg},\sigma\big)$ to $\big(k\langle b\rangle^{\alg},\sigma\big)$.
\end{itemize}
\end{proposition}

\begin{proof}
(i)$\implies$(ii) is clear.

(ii)$\implies$(iii).
Work in a sufficiently saturated elementary extension $(L,\partial)$ of $(K,\partial)$.
Then $(L,\sigma)$ is also saturated as a difference-field, and so $\tp_{\sigma}\big(\Theta a/k\big)=\tp_{\sigma}\big(\Theta b/k\big)$ is witnessed by a difference-field automorphism $\alpha$ of $L$ over $k$, taking $\Theta a$ to $\Theta b$.
Then $\beta:=\alpha\upharpoonright_{k[\Theta a]}$ is the desired $\cD$-field isomorphism from $k\langle a\rangle=k[\Theta a]$ to $k\langle b\rangle=k[\Theta b]$,
and $\alpha\upharpoonright_{k\langle a\rangle^{\alg}}$ is the desired extension.

(iii)$\implies$(i).
First note that the difference-field isomorphism, $\alpha$, from $k\langle a\rangle^{\alg}$ to $k\langle b\rangle^{\alg}$ will necessarily be a $\cD$-field isomorphism.
Indeed, $\alpha$ will take $\partial\upharpoonright_{k\langle a\rangle^{\alg}}$ to a $\cD$-field structure on $k\langle b\rangle^{\alg}$ whose associated endomorphism is $\sigma\upharpoonright_{k\langle a\rangle^{\alg}}$.
But by the uniqueness part of Lemma~\ref{extendtoalg}, this new $\cD$-structure must co-incide with $\partial\upharpoonright_{k\langle b\rangle^{\alg}}$.
The equality of types is now an immediate consequence of~\ref{overclosed} and~\ref{algsubstructure}.
\end{proof}

The equivalence of parts~(i) and~(iii) above yields the following corollary:

\begin{corollary}[Quantifier Reduction]
Every $L$-formula $\phi(x_1,\dots,x_n)$ is equivalent modulo $\ecdf$ to an $L$-formula of the form $\exists y \ \psi(x_1,\dots,x_n,y)$ where
\begin{itemize}
\item
$\psi(x_1,\dots,x_n,y)=\xi(\bar x,\bar y)$ where $\xi$ is a quantifier-free ring formula, the co-ordinates of $\bar x$ are of the form $\theta x_i$ where $\theta \in \Theta$ and $\bar{y} = (y,\sigma_1(y),\ldots,\sigma_t(y))$,
\item
each disjunct of $\xi$ written in disjunctive normal form includes a conjunct of the
form $t_N(\bar x) \neq 0 ~\&~ \sum_{j=0}^N t_j(\bar x) y^j = 0$ where each $t_i$ is a polynomial.
\end{itemize}
In particular, when the associated endomorphisms are all trivial the existential quantifier may be omitted and we have quantifier elimination.
\end{corollary}

\medskip
\subsection{Independence and simplicity}

\noindent
In this section we observe that  $\ecdf$ is simple, and we give an algebraic characterisation of nonforking independence.
The results here follow more or less axiomatically from the results of the previous sections, as established by Chatzidakis and Hrushovski in~\cite{acfa1}.

Let $(\mathbb U,\partial)$ be a sufficiently saturated model of $\ecdf$.

\begin{definition}
Suppose $A, B, C$ are (small) subsets of $\mathbb U$.
Then {\em $A$ is independent from $B$ over $C$}, denoted by $A\ind_CB$, if $\acl(A\cup C)$ is algebraically independent (equivalently linearly disjoint) from $\acl(B\cup C)$ over $\acl(C)$.
\end{definition}

\begin{theorem}
Independence in $(\mathbb U,\partial)$ satisfies the following properties:
\begin{itemize}
\item[(a)]
Symmetry.
$A\ind_CB$ implies $B\ind_CA$.
\item[(b)]
Transitivity.
Given $A\subseteq B\subseteq C$ and tuple $a$,
\begin{center}
$a\ind_AC$ if and only if $a\ind_BC$ and $a\ind_AB$.
\end{center}
\item[(c)]
Invariance.
If $\alpha\in\aut(\mathbb U,\partial)$ then $A\ind_CB$ implies $\alpha(A)\ind_{\alpha(C)}\alpha(B)$.
\item[(d)]
Finite character.
$A\ind_CB$ if and only if $A\ind_CB_0$ for all finite $B_0\subset B$.
\item[(e)]
Local character.
Given a set $B$ and a tuple $a$, there exists countable $B_0\subset B$ such that $a\ind_{B_0}B$.
\item[(f)]
Extension.
Given $A\subseteq B$ and tuple $a$, there exists a tuple $a'$ such that $\tp(a/A)=\tp(a'/A)$ and $a'\ind_AB$.
\item[(g)]
Independence theorem.
Suppose
\begin{itemize}
\item
$F$ is an algebraically closed inversive $\cD$-field,
\item
$A$ and $B$ are supersets of $F$ with $A\ind_FB$,
\item
$a\ind_FA$ and $b\ind_FB$
\item
$\tp(a/F)=\tp(b/F)$.
\end{itemize}
Then there is $d\ind_FAB$ with $\tp(d/A)=\tp(a/A)$ and $\tp(d/B)=\tp(b/B)$.
\end{itemize}
In particular, $\th(\mathbb U,\partial)$ is simple and $\ind$ is nonforking independence.
\end{theorem}

\begin{proof}
(a) through (e) follow easily from the corresponding properties for algebraic independence; part~(e) using also the fact that if $K$ is an inversive  $\cD$-field then $K\langle a\rangle$ is countably generated as a field over $K$.

(f).
Let $F=\acl(A)$, $K=\acl(B)$, and $K_1:=F\langle a\rangle^{\alg}$.
Let $K_1'$ be a field-isomorphic copy of $K_1$ over $F$ -- say with $\alpha:K_1\to K_1'$ witnessing this --  such that $K_1'$ is linearly disjoint from $K$ over $F$.
We can put a $\cD$-field structure $\partial'$ on $K_1'$ extending $(F,\partial)$ such that $\alpha$ is a $\cD$-field isomorphism.
Now by Lemma~\ref{amalgam} we can find a model of $\ecdf$ extending both $(K_1',\partial')$ and $(K,\partial)$.
By Proposition~\ref{overclosed} and saturation we may assume this model is an elementary substructure of $(\mathbb U,\partial)$.
Hence $\tp(\alpha(a)/F)=\tp(a/F)$ by the equivalence of~(i) and~(iii) in Proposition~\ref{types}, and $\alpha(a)\ind_AB$ by linear disjointedness.

(g).
We follow the spirit of the argument used for $ACFA$ in~\cite{acfa1}.
Fix $c\models p(x):=\tp(a/F)=\tp(b/F)$.
It suffices to find $A',B'$ such that
\begin{itemize}
\item[(i)]
$\{A',B',c\}$ is independent over $F$,
\item[(ii)]
$A'c\models \tp(Aa/F)$,
\item[(iii)]
$B'c\models \tp(Bb/F)$, and
\item[(iv)]
$A'B'\models \tp(AB/F)$.
\end{itemize}
Indeed, if $\alpha\in\aut_F(\mathbb U,\partial)$ with $\alpha(A'B')=AB$, then $d:=\alpha(c)$ will witness the conclusion.

Since $\tp(c/F)=\tp(a/F)=\tp(b/F)$, there exists $A'B'$ satisfying~(ii) and~(iii).
Moreover, by extension, we may also assume that $A'\ind_{Fc}B'$.
Hence, by transitivity, we have~(i) as well.
The only thing missing is~(iv).

Let $K_0:=\acl(A')\cdot\acl(B')$, and $K_1:=\acl(A'c)\cdot\acl(B'c)$, and $K_2:=\acl(A'B')$.
So $K_1$ and $K_2$ are field extensions of $K_0$.
We wish to give $K_2$ a $\cD$-field structure $\gamma$ such that
\begin{equation}
\label{iv}
(K_2,\gamma) \approx_F (\acl(AB),\partial\upharpoonright_{\acl(AB)}).
\end{equation}
To do so, denote by $\alpha$ and $\beta$ the $F$-automorphisms of the universe taking $A$ to $A'$ and $B$ to $B'$, respectively.
Then since $A\ind_FB$ and $A'\ind_FB'$, $\alpha\upharpoonright_{\acl(A)}\otimes\beta\upharpoonright_{\acl(B)}$ induces an isomorphism over $F$ between the fields $\acl(A)\cdot\acl(B)$ and $\acl(A')\cdot\acl(B')$, and hence between their field-theoretic algebraic closures $\big(\acl(A)\cdot\acl(B)\big)^{\alg}=\acl(AB)$ and $\big(\acl(A')\cdot\acl(B')\big)^{\alg}=\acl(A'B')=K_2$.
We use this field isomorphism to define the desired $\gamma$ on $K_2$ such that~(\ref{iv}) holds.

Since $\alpha\upharpoonright_{\acl(A)}$ and $\beta\upharpoonright_{\acl(B)}$ are $\cD$-field ismorphisms, we have that $\gamma$ agrees with $\partial$ on each of $\acl(A')$ and $\acl(B')$.
Hence $\gamma$ must agree with $\partial$ on the composite $K_0$.
That is, $(K_1,\partial\upharpoonright_{K_1})$ and $(K_2,\gamma)$ are $\cD$-field extensions of $(K_0,\partial\upharpoonright_{K_0})$.
If we can find a common extension $\tau$ of $\partial\upharpoonright_{K_1}$ and $\gamma$ to the composite $K_1\cdot K_2$, then we could extend $(K_1\cdot K_2,\tau)$ to a model of $\ecdf$ which will be elementarily embeddable in $(\mathbb U,\partial)$ over $F$ by Proposition~\ref{overclosed}.
We will thus have achieved~(iv) because of~(\ref{iv}), without ruining~(i) through~(iii), thereby proving the independence theorem.

To find such an extension, by Lemma~\ref{amalgam}, it suffices to show that $K_1$ and $K_2$ are linearly disjoint over $K_0$.
This follows from the following field-theoretic fact proved by Chatzidakis and Hrushovski in~\cite{acfa1}:
{\em If $A,B,C$ are algebraically closed fields extending an algebraically closed field $F$, with $C$ algebraically independent from $AB$ over $F$, then $(AC)^{\alg}(BC)^{\alg}$ is linearly disjoint from $(AB)^{\alg}$ over $AB$.}\footnote{See Remark~2 following the proof of the Generalised Independence Theorem in~\cite{acfa1}.}
\end{proof}

\begin{definition}[Dimension]
\label{dimtheta}
Suppose $a$ is a tuple and $k$ is an algebraically closed inversive $\cD$-subfield.
We let $\dim_{\cD}(a/k):=(\trdeg(\Theta_r(a)/k):r<\omega)$ where
$\Theta_r(a):=\big(\theta a:\theta\text{ a word of length $\leq r$ on }\{\partial_1,\dots,\partial_\ell,\sigma_1^{-1},\dots,\sigma_t^{-1}\} \big)$.
We view $\dim_{\cD}(a/k)$ as an element of $\omega^{\omega}$ equipped with the lexicographic ordering.
\end{definition}

Note that this dimension is not preserved under interdefinability, and that a more robust notion would depend only on the the eventual growth of the sequence of transcendence degrees.
This dimension should  be regarded as an analogue of the Kolchin function in differential algebra.   In some sense, it is too fine, but it will measure nonforking.

\begin{lemma}
\label{dimfork}
Suppose $a$ is a tuple and $k\subseteq L$ are algebraically closed inversive $\cD$-subfields.
Then $a\ind_kL$ if and only if $\dim_{\cD}(a/L)=\dim_{\cD}(a/k)$.
\end{lemma}

\begin{proof}
\begin{eqnarray*}
a\ind_kL
&\iff&
\acl(ka)\text{ is algebraically independent of $L$ over $k$} \ \ \ \ \text{(by definition)}\\
&\iff&
k(\Theta a)^{\alg}\text{ is algebraically independent of $L$ over $k$} \ \ \ \ \text{(by~\ref{acl})}\\
&\iff&
k\big(\Theta_r(a)\big)^{\alg}\text{ is algebraically independent of $L$ over $k$ for all $r<\omega$}\\
&\iff&
\trdeg(\Theta_r(a)/L)=\trdeg(\Theta_r(a)/k)\text{ for all $r<\omega$}\\
&\iff&
\dim_{\cD}(a/L)=\dim_{\cD}(a/k)
\end{eqnarray*}
 \end{proof}

\medskip
\subsection{Elimination of imaginaries}
\label{sectionei}

We follow the same basic strategy for proving elimination of imaginaries as that of Chatzidakis and Hrushovski in~\cite{acfa1}.

\begin{theorem}
$\th(\mathbb U,\partial)$ eliminates imaginaries.
\end{theorem}

\begin{proof}
The proof of elimination of imaginaries for $\operatorname{ACFA}$ given in~$\S1.10$ of~\cite{acfa1} actually proves that a simple theory admits weak elimination of imaginaries if {\em given any imaginary element $e=f(a)$, where $a$ is a tuple from the home sort and $f$ is a definable function, there exists $c\models\tp\big(a/\acl^{\eq}(e)\cap\mathbb U\big)$ with $f(c)=e$ and $c\ind_{\acl^{\eq}(e)\cap\mathbb U}a$.}
Since in any theory of fields weak elimination of imaginaries implies full elimination of imaginaries, it suffices to prove the existence of such a $c$.

Let $E:=\acl^{\eq}(e)\cap\mathbb U$.
As pointed out in~$\S1.10$ of~\cite{acfa1}, Neumann's Lemma implies that there exists $b\models\tp(a/Ee)$ with $\acl^{\eq}(Ea)\cap\acl^{\eq}(Eb)\cap\mathbb U=E$.
We first claim that such a $b$ can be chosen of maximal $\dim_{\cD}$ over $\acl(Ea)$ in the lexicographic ordering.
First, for each $r$, choose $b_r$ so that the above holds and $\big(\trdeg(\Theta_i(b_r/\acl(Ea)):i\leq r\big)$ is maximal possible.
Let $n_r:=\trdeg(\Theta_r(b_r)/\acl(Ea))$.
Note that for all $i\leq r$, $\trdeg(\Theta_i(b_r)/\acl(Ea))=n_i$.
Now let $\Phi(x)$ be the partial type over $\acl(Ea)$ saying that
\begin{itemize}
\item[]
$x\models\tp(a/Ee)$,
\item[]
$\acl^{\eq}(Ea)\cap\acl^{\eq}(Ex)\cap\mathbb U=E$, and,
\item[]
for each $r<\omega$, $\trdeg(\Theta_r(x)/\acl(Ea))\geq n_r$.
\end{itemize}
The $b_r$'s witness that $\Phi(x)$ is finitely satisfiable, and hence by compactness it is satisfiable.
Letting $b$ realise $\Phi(x)$ we have that  $b\models\tp(a/Ee)$, $\acl^{\eq}(Ea)\cap\acl^{\eq}(Eb)\cap\mathbb U=E$, and $\dim_{\cD}(b/\acl(Ea))$ is maximal.

Now we proceed as in~$\S1.10$ of~\cite{acfa1}.
Let $c\models\tp(b/\acl(Ea))$ with $c\ind_{Ea}b$.
Then $c\models\tp(a/Ee)$ and so $f(c)=e$.
So it remains to show that $c\ind_Ea$.

We have that $\acl^{\eq}(Ec)\cap\acl^{\eq}(Eb)\subseteq\acl^{\eq}(Ea)$ by independence, and hence
$\acl^{\eq}(Ec)\cap\acl^{\eq}(Eb)\cap\mathbb U\subseteq \acl^{\eq}(Ea)\cap\acl^{\eq}(Eb)\cap\mathbb U=E$.
Letting $c'$ be such that $\tp(bc/Ee)=\tp(ac'/Ee)$ we have that $c'\models\tp(a/Ee)$ and $\acl^{\eq}(Ec')\cap\acl^{\eq}(Ea)\cap\mathbb U=E$.
So by maximality, $\dim_{\cD}(c'/\acl(Ea))\leq\dim_{\cD}(b/\acl(Ea))$.
Hence, as $\dim_{\cD}$ is automorphism invariant, $\dim_{\cD}(c/\acl(Eb))\leq\dim_{\cD}(b/\acl(Ea))$.
But, on the other hand,
$$\dim_{\cD}(c/\acl(Eb))\geq\dim_{\cD}(c/\acl(Eab))=\dim_{\cD}(c/\acl(Ea))=\dim_{\cD}(b/\acl(Ea))$$
where the first equality is by Lemma~\ref{dimfork}.
Hence we have equality throughout, and $\dim_{\cD}(c/\acl(Eb))=\dim_{\cD}(c/\acl(Eab))$ which, by Lemma~\ref{dimfork} again, implies that $c\ind_{Eb}a$.
Since we also have $c\ind_{Ea}b$, and $\acl^{\eq}(Ea)\cap\acl^{\eq}(Eb)\cap\mathbb U=E$, we get $c\ind_Eab$.
In particular $c\ind_Ea$, as desired.
\end{proof}

\bigskip

\section{The Zilber dichotomy for finite-dimensional minimal types}
\label{zilbersec}

\noindent
In this final chapter we begin to study the fine structure of definable sets in $\ecdf$.
As is by now a standard approach, the first step is to prove a Zilber dichotomy type theorem for the types of $SU$-rank one as these form the building blocks of the finite rank definable sets.
A second step, which we do not carry out here, would be to consider regular types more generally.
Our current methods only allow us to handle the ``finite-dimensional" case.

We continue to work in a sufficiently saturated model $(\mathbb U,\partial)\models\ecdf$, and over a (small) inversive $\cD$-subfield $k$.

\begin{definition}
A type $p=\tp(a/k)$ is called {\em finite-dimensional} if the $\cD$-field generated by $a$ over $k$ is of finite transcendence degree over $k$.
\end{definition}

From Proposition~\ref{acl} we know that in general $\acl(ka)=k\langle a\rangle^{\alg}$, the field-theoretic algebraic closure of the {\em inversive} $\cD$-field generated by $a$ over $k$.
The anonymous referee of an earlier version of this paper pointed out to us that in the finite-dimensional case the inversiveness comes for free:

\begin{lemma}
\label{zoetip}
Suppose $\tp(a/k)$ is finite-dimensional.
Then $\acl(ka)$ is the field-theoretic algebraic closure of the $\cD$-field generated by $a$ over $k$.
\end{lemma}

\begin{proof}
Let $L$ be the $\cD$-field generated by $a$ over $k$.
So if $\Xi$ is the set of all finite words on $\{\partial_1,\dots,\partial_{\ell-1}\}$, then $L=k(\xi a:\xi\in\Xi)$.
We need to show that $L^{\alg}$ is already inversive.
Let $b\in L^{\alg}$, and fix one of the associated endomorphisms $\sigma_i$.
By finite-dimensionality, that is by the finite transcendence of $L^{\alg}$ over $k$, for some $r\geq 0$, $\{b,\sigma_ib,\dots,\sigma_i^rb\}$ is algebraically dependent over $k$. Applying $\sigma_i^{-1}$ sufficiently many times to the algebraic relation witnessing this, we get that $b\in k(\sigma_ib,\dots,\sigma_i^rb)^{\alg}$.
Applying $\sigma_i^{-1}$ one more time, we get $\sigma_i^{-1}b\in k(b,\sigma_ib,\dots,\sigma_i^{r-1}b)^{\alg}\subseteq L^{\alg}$.
So $L^{\alg}$ is inversive, as desired.
\end{proof}

\begin{corollary}
\label{zoecor}
Suppose $\tp(a/k)$ is finite-dimensional and let $\Xi$ be the set of all finite words on $\{\partial_1,\dots,\partial_{\ell-1}\}$.
For $L$ any inversive $\cD$-field extending $k$,
$$a\ind_kL \ \iff \ \trdeg\big(k(\xi a:\xi\in\Xi)/k\big)=\trdeg\big(L(\xi a:\xi\in\Xi)/L\big)$$
\end{corollary}

\begin{proof}
By Lemma~\ref{zoetip}, $\acl(ka)=k(\xi a:\xi\in\Xi)^{\alg}$.
So by definition, $a\ind_kL$ if and only if and only if $k(\xi a:\xi\in\Xi)^{\alg}$ is algebraically independent of $L^{\alg}$ over $k^{\alg}$.
Since $k(\xi a:\xi\in\Xi)$ is of finite transcendence degree over $k$, this last condition is equivalent to the one stated in the corollary.
\end{proof}

\begin{definition}
The {\em field of constants} is $C:=\{x\in\mathbb U:e(x)=s(x)\}$.
\end{definition}

The goal of this chapter is to prove that if $p$ is a finite-dimensional type of $\operatorname{SU}$-rank one, then either $p$ is one-based or it is almost internal to the field of constants.
We follow here the strategy of Pillay and Ziegler~\cite{pillayziegler03} by proving first a canonical base property (see~\ref{cbp} for a precise statement in our context) using an appropriate notion of {\em jet spaces}.
One such notion well-suited to the present context was developed in~\cite{paperB}, but some preliminaries are necessary to relate the formalisms of that paper and the current one.

Assumptions~\ref{assumptionA} and the notation of Chapter~\ref{chapter-mc} remain in place.

\medskip
\subsection{Iterativity}
We begin by describing how $\cD$ gives rise to a generalised iterative Hasse-Schmidt system in the sense of~\cite{paperB}.
The construction here is essentially the same as (though dual to) that of Kamensky (Proposition~2.3.2 of~\cite{kamensky}).

First of all, one can always form a completely free iterative Hasse-Schmidt system by simply iterating $\mathcal D$ with itself.
That is, one defines the projective system of finite free algebra schemes $(\cD^{(n)}, s_n,\psi_n,)$ by
\begin{itemize}
\item[]
$\cD^{(n+1)}(R)=\cD\big(\cD^{(n)}(R)\big)$
\item[]
$s_{n+1}^R:= s^{\cD^n(R)} \circ s_n^R:R\to\cD^{(n+1)}(R)$
\item[]
$\psi_{n+1}^R:=(\psi_n^R)^\ell\circ\psi^{\cD^{(n)}(R)}:\cD^{(n+1)}(R)\to (R^\ell)^{n+1}$
\item[]
$f_{n+1}^R:=\pi^{\cD^{(n)}(R)}:\cD^{(n+1)}(R)\to\cD^{(n)}(R)$
\end{itemize}
for any $A$-algebra $R$.
For details on composing finite free $\mathbb S$-algebras see $\S 4.2$ of~\cite{paperA}.
Equipped with the trivial iterativity maps (since $\cD^{(n+m)}=\cD^{(n)}\circ\cD^{(m)}$), this becomes a generalised iterative Hasse-Schmidt system (see~2.2 and ~2.17 of~\cite{paperB}).

However, the above construction does not take into account the fact that in our $\cD$-rings $(R,e)$, the coefficient of $\epsilon_0$ in $e(a)$ is always $a$.
In other words, the fact that $e$ is a section to $\pi=\pr_1\circ\psi$.
We will thus need to define a sequence of subalgebra schemes $\cD_n\subseteq \cD^{(n)}$ by identifying the appropriate co-ordinates.
This is done as follows.
Given an $A$-algebra $R$, fix the $R$-basis $\big\{\epsilon_{i_1}\otimes\cdots\otimes\epsilon_{i_n}:0\leq i_j\leq\ell-1\}$ for $\cD^{(n)}(R)$.
Define $(i_1,\dots,i_n)\sim(j_1,\dots,j_n)$ if $(i_1,\dots,i_n)$ and $(j_1,\dots,j_n)$ yield the same ordered tuple when all the zeros are dropped.
Then $\cD_n(R)$ is the subalgebra of elements
$$\left\{\displaystyle\sum r_{i_1,\dots,i_n}(\epsilon_{i_1}\otimes\cdots\otimes\epsilon_{i_n})\ | \ r_{i_1,\dots,i_n}=r_{j_1,\dots,j_n}\text{ whenever }(i_1,\dots,i_n)\sim(j_1,\dots,j_n)\right\}$$
It follows from this that $\psi_n:\cD^{(n)}\to\mathbb (S^\ell)^n$ maps $\cD_n$ onto the diagonal defined by equating the $(i_1,\dots,i_n)$th and $(j_1,\dots,j_n)$th co-ordinates whenever $(i_1,\dots,i_n)\sim(j_1,\dots,j_n)$.
This diagonal is canonically identified with the free $\mathbb S$-module scheme $\mathbb S^{L_n}$ where $L_n:=\{(i_1,\dots,i_m):0\leq m\leq n, 0< i_j\leq\ell-1\}$.

\begin{remark}
We can define $\cD_n$ in a co-ordinate free manner as follows.
Given $n>0$ and $1\leq i\leq n$, consider the morphism of algebra schemes
$$\lambda_{i,n}:=\cD^{i-1}(f_{n-i+1}):\cD^{(n)}\to\cD^{(n-1)}$$
Then $\cD_n$ is the equaliser of $\lambda_{1,n},\dots,\lambda_{n,n}$.
\end{remark}

The following properties follow:
\begin{itemize}
\item
$s_n:\mathbb S\to\cD^{(n)}$ maps $\mathbb S$ to $\cD_n$, so the latter becomes an $\mathbb S$-algebra scheme.
\item
$\psi_n:\cD^{(n)}\to\mathbb (S^\ell)^n$ maps $\cD_n$ to $\mathbb S^{L_n}$ isomorphically as
an ${\mathbb S}$-module
\item
$f_n:\cD^{(n)}\to\cD^{(n-1)}$ restricts to a surjective morphism of $\mathbb S$-algebra schemes from $\cD_n$ to $\cD_{n-1}$.
\item
As subalgebra schemes of $\cD^{(m+n)}$, $\cD_{m+n}\subseteq\cD_m\circ\cD_n$.
\end{itemize}
Hence, $\underline\cD:=(\cD_n)$ is a generalised iterative Hasse-Schmidt system.

\begin{proposition}
\label{etoE}
Suppose $(R,e)$ is a $\cD$-ring.
Then there is a unique iterative $\underline\cD$-ring structure $E= \big(E_n:R\to\cD_n(R):n<\omega\big)$ on $R$ with $E_1=e$.
In terms of co-ordinates, this $\underline\cD$-ring structure is given by
\begin{eqnarray}
\label{enop}
E_n(a)
&=&
\sum_{(i_1,\dots,i_m)\in L_n}\partial_{i_1}\cdots\partial_{i_m}(a)\epsilon_{i_1,\dots,i_m}
\end{eqnarray}
where $L_n:=\{(i_1,\dots,i_m):0\leq m\leq n, 0< i_j\leq\ell-1\}$ and $\{\epsilon_{i_1,\dots,i_m}\}$
is the $R$-basis for $\cD_n(R)$ obtained from the standard basis for $R^{L_n}$ via $\psi_n^R$.
\end{proposition}

\begin{proof}
First of all, recall that $E= \big(E_n:R\to\cD_n(R):n<\omega\big)$ is an iterative $\underline\cD$-ring structure on $R$, according to Definitions~2.2 and ~2.17 of~\cite{paperB}, if the maps are all ring homomorphisms and
\begin{itemize}
\item[(i)]
$E_0=\id$
\item[(ii)]
$f_{m,n}^R\circ E_m=E_n$ for all $m\geq n$, where $f_{m,n}:\cD_m\to\cD_n$ is $f_m\circ\cdots\circ f_{n+1}$,
\item[(iii)]
$E_{m+n}=\cD_m(E_n)\circ E_m$ for all $m,n$.
\end{itemize}
For existence, we define $E_n:R\to\cD^{(n)}(R)$ by composing $e$ with itself $n$-times.
That is, recursively, $E_0=\id$ and $E_{n+1}=\cD(E_n)\circ e$.
That $E_n$ maps $R$ to $\cD_n(R)$, and that it has the form claimed in~(\ref{enop}), is not difficult to check using the fact that $e(a)=a+\partial_1(a)\epsilon+\cdots+\partial_{\ell-1}(a)\epsilon_{\ell-1}$. Properties~(i) through~(iii) follow more or less immediately from~(\ref{enop}).

For uniqueness, note that the assumption that $E_1=e$ and property~(iii) force $E_{n+1}=\cD(E_n)\circ e$, so that the above construction is the only one possible.
\end{proof}

\medskip

\subsection{\underline{$\cD$}-varieties and generic points}
We return now to our saturated model $(\mathbb U,\partial)$ of $\ecdf$, and equip it with the definable iterative $\underline\cD$-field structure $E$ given by Proposition~\ref{etoE}.
We fix also an inversive $\cD$-subfield $k$.

In Section~3 of~\cite{paperB} the rudiments of $\underline\cD$-algebraic geometry are developed.
Let us recap some of the notions introduced there.
We work inside a fixed irreducible algebraic variety $X$ over $k$.
While the treatment in~\cite{paperB} is more general, for the sake of concreteness and also with an eye toward our model-theoretic intentions, we will assume that $X$ is affine.
By a {\em $\underline\cD$-subvariety of $X$ over $k$} is meant a sequence $\underline Z=(Z_n)_{n<\omega}$ of subvarieties $Z_n\subseteq \tau_n X$ over $k$ such that for all $0<n<\omega$,
\begin{itemize}
\item[(1)]
$\hat f_n:\tau_nX\to\tau_{n-1}X$ restricts to a morphism from $Z_n\to Z_{n-1}$, and
\item[(2)]
$Z_n\subseteq\tau(Z_{n-1})$.
\end{itemize}
Here, $\tau_n X=\tau(X,\cD_n,E_n)$ is the $n$th prolongation of $X$ in the sense of~$\S2.1$ of~\cite{paperB}.
In particular, $\tau_nX(\mathbb U)$ is canonically identified with $X\big(\cD_n(\mathbb U)\big)$.

\begin{remark}
The iterativity condition of Definition~3.1 of~\cite{paperB} here simplifies to~(2) above because the iterativity map here is just the containment $\cD_n\subseteq\cD\circ\cD_{n-1}$ as subschemes of $\cD^{(n)}$.
\end{remark}

The map $E_n$ induces a definable map $\nabla_n:X(\mathbb U)\to\tau_nX(\mathbb U)$ which with respect to the standard bases is given by
$$\nabla_n(p):=\big(\partial_{i_1}\cdots\partial_{i_m}(p):0\leq m\leq n, 0< i_j\leq\ell-1\big)$$
We have seen this map appear already in the proof of Theorem~\ref{ecedomains}.

We say $\underline Z$ is ({\em absolutely}) {\em irreducible} if each $Z_n$ is (absolutely) irreducible.

By the {\em ${\mathbb U}$-rational points} (or simply the \emph{rational points}) of $\underline Z$
is meant the type-definable set
$$\underline Z(\mathbb U):=\{p\in X(\mathbb U):\nabla_n(p)\in Z_n(\mathbb U),\text{ for all }n<\omega\}$$
A rational point $p\in\underline Z(\mathbb U)$ is {\em $k$-generic} if $\nabla_n(p)$ is generic in $Z_n$ over $k$ in Weil's sense that
there is no proper $k$-subvariety $Y \subsetneq Z_n$ with $\nabla_n(p) \in Y({\mathbb U})$), for all $n$.

Of course nothing so far has guaranteed the existence of generic points, or even of rational points.
In~\cite{paperB} this is dealt with by working in ``rich'' $\underline\cD$-fields.
Here we do not have richness of $(\mathbb U,\partial)$, however, we can characterise precisely which $\underline\cD$-varieties do have generic points.
Indeed, what is required is the following higher-order analogue of the condition appearing in axiom~III of $\ecdf$ (cf.~Theorem~\ref{ecedomains}).

\begin{definition}[$\sigma$-dominance]
Suppose $\underline Z=(Z_n)_{n<\omega}$ is a $\underline\cD$-subvariety of an algebraic variety $X$.
We will say that $\underline Z$ is {\em $\sigma$-dominant} if for each $n>0$ and each $i=0,\dots,t$, the morphism $\hat\pi_i:\tau(Z_{n-1})\to Z_{n-1}^{\sigma_i}$ restricts to a dominant morphism from $Z_n$ to $Z_{n-1}^{\sigma_i}$.
(See $\S$\ref{mc} to recall what $\pi=\pi_0,\dots,\pi_t$ are.)
\end{definition}

\begin{remark}
\label{jetbdominance}
It follows from $\sigma$-dominance that $\hat f_n\upharpoonright_{Z_n}:Z_n\to Z_{n-1}$ is dominant.
Indeed, this is because $\pi_0=\pi$ and $f_n$ is just $\pi$ applied to $\cD^{(n)}$.
Hence a $\sigma$-dominant $\underline\cD$-variety is in particular {\em dominant} in the sense of~\cite{paperB}.
\end{remark}

\begin{proposition}
\label{genpoints}
Suppose $\underline Z$ is an absolutely irreducible $\underline\cD$-subvariety of $X$ over~$k$.
Then $\underline Z$ has a $k$-generic point if and only if $\underline Z$ is $\sigma$-dominant.
\end{proposition}

\begin{proof}
First suppose that $\underline Z$ has a $k$-generic point $p\in\underline Z(\mathbb U)$.
Then $\nabla_n(p)$ is generic in $Z_n(\mathbb U)$.
To prove $\sigma$-dominance we will show that $\hat\pi_i\big(\nabla_n(p)\big)$ is generic in $Z_{n-1}^{\sigma_i}$ over $k$.
Under the identification of $\tau Z_{n-1}(\mathbb U)$ with $Z_{n-1}\big(\cD(\mathbb U)\big)$, $\nabla_n(p)$ corresponds $e\big(\nabla_{n-1}(p)\big)$, and we need to show that $\pi_i\big(e(\nabla_{n-1}(p))\big)=\sigma_i\big(\nabla_{n-1}(p)\big)$ is generic in $Z_{n-1}^{\sigma_i}$ over $k$.
But this follows from the fact that $\nabla_{n-1}(p)$ is generic in $Z_{n-1}$ and $\sigma_i$ is an automorphism.

For the converse we assume that $\underline Z$ is $\sigma$-dominant and seek a generic rational point.
Without loss of generality we may assume that $k$ is algebraically closed.
By saturation it suffices to fix $n<\omega$ and show that there exists $p\in X(\mathbb U)$ such that $\nabla_n(p)$ is generic in $Z_n$ over $k$.
For this we follow the general strategy in the proof of Theorem~\ref{ecedomains}.
Let $L$ be an algebraically closed field extension of $k$, $b\in Z_n(L)$ a generic point over $k$, and $a:=\hat f_{n,0}(b)\in X(L)$.
We will show how to extend $\partial$ from $k$ to a $\cD$-field structure on some extension of $L$, such that $\nabla_n(a)=b$.
This $\cD$-field structure could then be further extended to a model of $\ecdf$, which by Proposition~\ref{overclosed} and saturation can then be embedded into $(\mathbb U,\partial)$ over $k$; thus establishing the existence of the desired $p\in X(\mathbb U)$.

Under the identification of $\tau Z_{n-1}(L)$ with $Z_{n-1}\big(\cD(L)\big)$, let $b'$ be the tuple from $\cD(L)$ corresponding to $b$.
We have that $P^e(b')=0$ for all $P(x)\in I(Z_{n-1}/k)$.
On the other hand, as $Z_n\to Z_{n-1}$ is dominant (cf.  Remark~\ref{jetbdominance}), $I(Z_{n-1}/k)=I(b_{n-1}/k)$.
It follows that $e$ on $k$ extends to a ring homomorphism $\eta:k[b_{n-1}]\to\cD(L)$, where $b_{n-1}$ is the image of $b$ under $Z_n\to Z_{n-1}$, by $\eta(b_{n-1})=b'$.
The assumption of $\sigma$-dominance implies that $\pi_i\circ\eta:k[b_{n-1}]\to L$ is injective, for each $i=1,\dots,t$, so that by Lemmas~\ref{extendtofields} and~\ref{extendtoautos} we can extend $\eta$ to a $\cD$-field structure on some field $L'$ extending $L$.
In this $\cD$-field, the fact that $\eta(b_{n-1})=b'$ means that $\nabla(b_{n-1})=b$.
It then follows, that for any $r<n$,
\begin{eqnarray*}
\nabla\big(\hat f_{n,r}(b)\big)
&=&
\nabla\big(\hat f_{n-1,r}(b_{n-1})\big)\\
&=&
\hat f_{n,r+1}\big(\nabla(b_{n-1})\big)\ \ \ \ \text{ cf. Proposition~4.7(a) of~\cite{paperA}}\\
&=&
\hat f_{n,r+1}(b)
\end{eqnarray*}
Iterating, and recalling that $\hat f_{n,0}(b)=a$ and $\hat f_{n,n}(b)=b$, we get that $\nabla_n(a)=b$, as desired.
\end{proof}

\medskip
\subsection{Jet spaces}
In~\cite{paperA} and~\cite{paperB}, following the work of Pillay and Ziegler~\cite{pillayziegler03}, we effected a linearisation of generalised Hasse-Schmidt varieties by introducing {\em jet spaces}.
We now specialise this theory to our present context (Fact~\ref{jetsdetermine} below), and also prove a finiteness theorem (Proposition~\ref{smalljet} below) that was not done in the earlier papers but is essential here.
We continue to work in a fixed affine algebraic variety $X$ over an inversive $\cD$-subfield $k$.

To each point $p$ of $X$, and for each $m>0$, we can associate a linear algebraic variety called the {\em $m$th algebraic jet space of $X$ at $p$}, denoted by $\jet^mX_p$.
It is a kind of higher-order tangent space; see~$\S 5$ of~\cite{paperA} for a review of this notion.
Now suppose that  $\underline Z=(Z_n)_{n<\omega}$ is a $\underline\cD$-subvariety of $X$ over $k$ and $p\in\underline Z(\mathbb U)$.
One would like to associate a $\underline\cD$-jet space to $\underline Z$ at $p$.
A natural thing would be to consider $\big(\jet^m(Z_n)_{\nabla_n(p)}\big)_{n<\omega}$.
However, this sequence does {\em not} determine a $\underline\cD$-subvariety of $\jet^mX_p$, simply because $\jet^m(Z_n)_{\nabla_n(p)}$ lives in $\jet^m(\tau_n X)_{\nabla_n(p)}$ rather than in $\tau_n(\jet^m X_p)$ as would be required by the definition of $\underline\cD$-subvariety.
We addressed this issue in~\cite{paperA} and~\cite{paperB} by studying a certain canonical linear {\em interpolating map} $\phi:\jet^m(\tau_nX)_{\nabla_n(p)}\to\tau_n(\jet^mX_p)$.
See $\S2.3$ of~\cite{paperB} for a brief description of $\phi$.
In~$\S4$ of~\cite{paperB} we were then able to define the {\em $m$th $\underline\cD$-jet space of $\underline Z$ at $p$}, denoted by $\jet^m_{\underline\cD}(\underline Z)_p$, as the $\cD$-subvariety defined by taking the Zariski closure of the images of the $\jet^m(Z_n)_{\nabla_n(p)}$ under these interpolating maps.
In particular we show (Lemma~4.4 of~\cite{paperB}) that for generic $p$,
\begin{equation}
\label{djet}
\jet_{\underline\cD}^m(\underline Z)_p(\mathbb U)=\left\{\lambda\in\jet^mX_p(\mathbb U):\nabla_n(\lambda)\in\phi\big(\jet^m(Z_n)_{\nabla_n(p)}(\mathbb U)\big),\forall n\geq 0\right\}
\end{equation}
As $\jet^m(Z_n)_{\nabla_n(p)}(\mathbb U)$ is a $\mathbb U$-linear subspace of $\jet^m(\tau_nX)_{\nabla_n(p)}(\mathbb U)$, $\phi$ is $\mathbb U$-linear, and $\nabla_n$ is $C$-linear, we get that $\jet_{\underline\cD}^m(\underline Z)_p(\mathbb U)$ is a $C$-linear subspace of $\jet^mX_p(\mathbb U)$.
Moreover, $\jet_{\underline\cD}^m(\underline Z)_p$ is the fibre above $p$ of a bundle $\jet_{\underline\cD}^m(\underline Z)\to\underline Z$ which is a $\underline\cD$-subvariety of the algebraic jet bundle $\jet^mX\to X$.
We refer the reader to~\cite{paperB} for more details on these spaces.

One thing that will be important for us is that if $\underline Z$ is $\sigma$-dominant then so is $\jet_{\underline\cD}^m(\underline Z)$.
Indeed, in Proposition~4.7 of~\cite{paperB} it is proved that the dominance of the $\hat\pi_0$ maps are preserved when one takes jets, and the very same proof works for the other $\hat\pi_i$ maps as well -- the key lemma behind all these cases being the ``compatibility of the interpolating map with comparing of prolongations'' which is~6.4(c) of~\cite{paperA}.
Similarly, the proof of~4.7 of~\cite{paperB} also shows that, if $p$ is a $k$-generic rational point of $\underline Z$, and $\underline Z$ is $\sigma$-dominant, then so is $\jet_{\underline\cD}^m(\underline Z)_p$.

We need one more piece of notation before stating the main result of~\cite{paperB} specialised to the present context.
Given $\underline Z=(Z_n)_{n<\omega}$ and $r<\omega$, we let
$\nabla_r\underline Z:=(Z_{r+n})_{n<\omega}$.
It is clear that $\nabla_r\underline Z$ is a $\underline\cD$-subvariety of $Z_r$.
If $\underline Z$ is $\sigma$-dominant, then this is also the case for $\nabla_r\underline Z$.
Finally, assuming $\sigma$-dominance, one can show that the set of rational points  of  $\nabla_r\underline Z$ is exactly $\nabla_r\big(\underline Z(\mathbb U)\big)$, see the proof\footnote{Actually the proof in this case is much easier as the iterativity maps are trivial.} of~3.16 of~\cite{paperB}.

\begin{fact}
\label{jetsdetermine}
Suppose $L$ and $L'$ are inversive $\cD$-subfields extending $k$, $\underline Z$ and $\underline Z'$ are absolutely irreducible $\underline\cD$-subvarieties of $X$ over $L$ and $L'$ respectively, and $p\in \underline{Z}(\mathbb U)\cap\underline{Z'}(\mathbb U)$ is $L$-generic in $\underline Z$ and $L'$-generic in $\underline Z'$.
If $\jet_{\underline\cD}^m(\nabla_r\underline Z)_{\nabla_r(p)}(\mathbb U)=\jet_{\underline\cD}^m(\nabla_r\underline Z')_{\nabla_r(p)}(\mathbb U)$ for all $m\geq 1$ and $r\geq 0$, then $\underline Z=\underline Z'$.
\end{fact}

\begin{proof}
This follows from the proof of Theorem~4.8 of~\cite{paperB} specialised to our context.
The only reason we cannot apply that theorem directly is because the ``richness'' assumptions of that theorem are not necessarily satisfied here.
However, richness is only used in the proof to ensure that all the relevant $\underline\cD$-varieties have (densely) many rational points.
Hence, because of Proposition~\ref{genpoints} of this paper, all one needs for that proof to go through is that the relevant $\underline\cD$-varieties here be $\sigma$-dominant.
This is the case for $\underline Z$ and $\underline Z'$ because by assumption they have a generic rational point, and as discussed in the preceding paragraphs taking jets and $\nabla_r$ preserves $\sigma$-dominance.

Let us also remark that Theorem~4.8 of~\cite{paperB} asks for $p$ to be in the ``good locus" of $\underline Z$ (and also of $\underline Z'$).
That is, $p$ should be smooth on $X$, $\hat f_n$ restricted to $Z_n$ should be smooth at $\nabla_n(p)$, and also $\nabla_n(p)$ should land inside a certain $L$-definable nonempty Zariski open subset of $Z_n$ mentioned in Lemma~4.4 of~\cite{paperB}, for all $n>0$.
Our assumption here that $p$ is $L$-generic (recalling also that we are in characteristic zero) ensures that $p$ is in the good locus.
\end{proof}

As one might expect, the linearisation that Fact~\ref{jetsdetermine} gives us is particularly useful when the $\underline\cD$-jet spaces are finite dimensional as vector spaces over the constants.
In the rest of this section we aim to prove that if $\underline Z=(Z_n)_{n<\omega}$ is ``finite-dimensional" in the sense that $\dim Z_n$ is bounded independently of $n$, then the jet space $\jet_{\underline\cD}^m(\underline Z)_p(\mathbb U)$ is a finite dimensional $C$-vector space.  Towards this
end we begin with a study of prolongations of schemes defined over the constants.  With our first technical lemma we observe that the constants are
constants for all of the higher exponentials as well.

\begin{lemma}
\label{iterconstant}
The maps $E_n$ and $s_n$ agree on $C$ for every natural number $n$.
\end{lemma}

\begin{proof}
This lemma follows by induction and the iterative construction, once one observes that the following diagram commutes for any $A$-algebra $R$
$$
\xymatrix{
R \ar[d]_{s_n}\ar[rr]^{s} & & \cD(R) \ar[d]^{\cD(s_n)}\\
\cD_n(R)\ar[rr]_{s^{\cD_n(R)} \ \ \ } & & \cD\big(\cD_n(R)\big)
}
$$
To see that the above diagram does indeed commute one writes it in terms of tensor products as
$$
\xymatrix{
R \ar[d]_{s_n}\ar[rr]^{s} & & R\otimes_A\cD(A) \ar[d]^{s_n\otimes_A\cD(A)}\\
R\otimes_A\cD_n(A)\ar[rr]_{s^{\cD_n(R)} \ \ \ \ \ \ } & & R\otimes_A\cD_n(A)\otimes_A\cD(A)
}
$$
which commutes because all the maps in this diagram are the natural algebra structure maps.
\end{proof}

In Section 4.1 of~\cite{paperA}  maps between prolongation spaces associated to maps of $\bbs$-algebras are
constructed.  More precisely, given a map $\alpha:\cE \to \cF$ of finite free $\bbs$-algebras, an $A$-algebra
$k$, and $\cE$-ring and $\cF$-ring structures $e:k \to \cE(k)$ and $f:k \to \cF(k)$ on $k$ for which the diagram
$$\xymatrix{ \cE(k) \ar[rr]^{\alpha^k} & & \cF(k) \\ & k \ar[ul]_{e} \ar[ur]^{f} & }$$ commutes,
for any $k$-scheme $X$, there is a map $\widehat{\alpha}:\tau(X,\cE,e) \to \tau(X,\cF,f)$ so that for any $k$-algebra $R$ the
diagram
$$\xymatrix{\tau(X,\cE,e) \ar[rr]^{\widehat{\alpha}} \ar@{=}[d] && \tau(X,\cF,f) \ar@{=}[d] \\
X(\cE^e(R)) \ar[rr]^{\alpha} && X(\cF^f(R)) }$$
commutes where the vertical arrows come from the natural identifications.
  It is clear from this interpretation that if
$\beta:\cF \to \cG$ is another map of finite free-$\bbs$-algebras and $g:k \to \cG(k)$ makes $k$ into a $\cG$-ring
with $\beta^k \circ f = g$, then $\widehat{\beta \circ \alpha} = \widehat{\beta} \circ \widehat{\alpha}$.

Specialising to $k = C$ and $\alpha=s_n:\bbs \to \cD_n$, we obtain a {\em constant} section map $z_n := \widehat{s_n}: X = \tau(X,\bbs,\operatorname{id}_k) \to
\tau(X,\cD_n,E_n) = \tau_n X$.   By Proposition 4.18(c) of~\cite{paperA}, the map $z_n:X \to \tau_n X$ is a closed immersion.   That
$\underline{z X} := (z_n X)_{n < \omega}$ is a $\sigma$-dominant $\underline{\cD}$-subscheme of $X$ follows from the functoriality
in the $\bbs$-algebra of the prolongation space construction.  For example, from the commutativity of the diagram

$$
\xymatrix{ && \cD_{n+1} \ar[dd]_{f_n} \\
\bbs \ar[urr]^{s_{n+1}} \ar[drr]_{s_n} & & \\
&& \cD_n }
$$
we deduce that $\widehat{f}_n$ takes $z_{n+1} X$ to $z_n X$.

Likewise, from the commutativity of

$$
\xymatrix{
\cD_n \ar[rr]^{s_m^{\cD_n}} && \cD_m \circ \cD_n \\
& \cD_{n+m}  \ar@{^{(}->}[ur] & \\
\bbs \ar[ur]^{s_{n+m}} \ar[uu]^{s_n} \ar[rr]^{s_m} & & \cD_m \ar[uu]^{\cD_m(s_n)}
}
$$
and the  fact that $z_m$ is a natural transformation (see Proposition 4.18 of~\cite{paperA}), we obtain the following commutative diagram

$$
\xymatrix{
\tau_n X \ar[rr]^{z_{m,\tau_n X}} && \tau_m \tau_n X \\
z_n X \ar@{^{(}->}[u] \ar[rr]^{z_{m,z_n X}} && \tau_m z_n X \ar@{^{(}->}[u] \\
 & \tau_{m+n} X \ar@{^{(}->}[uur] & \\
 X \ar[uu]^{z_n} \ar[rr]^{z_m} \ar[ur]^{z_{n+m}} && \tau_{m} X \ar[uu]^{\tau_m(z_n)}
}
$$
from which we deduce that via the inclusion of $\tau_{n+m} X$ in $\tau_m \tau_n X$, $z_{n+m} X$ is
contained in $\tau_m z_n X$.   Thus, $\underline{z X}$ is a $\underline{\cD}$-subscheme of $X$.

Finally, to check that this $\underline{\cD}$-scheme is
$\sigma$-dominant, observe that if $\pi_i:\cD \to \bbs$ is a projection map corresponding to one of the distinguished endormorphisms $\sigma_i$, the commutative diagram of $\bbs$-algebras  $\xymatrix@1{\bbs \ar@/^/[r]^{s} & \cD \ar@/^/[l]_{\pi_i} }$ shows that $\hat{\pi}_i \circ z = \operatorname{id}$.
It follows that for every $n>0$, the restriction of $\hat{\pi}_i$ to $z_nX=z(z_{n-1}X)$ is dominant onto $z_{n-1}X = (z_{n-1} X)^{\sigma_i}$, where the final equality is because $\sigma_i$ is the identity on $C$.

With the next lemma we show that the $\mathbb U$-rational points of $\underline{z X}$ are exactly the $C$-rational points of $X$.

\begin{lemma}
\label{zerosect}
$\underline{z X} (\mathbb U)=X(C)$
\end{lemma}

\begin{proof}
Suppose $p\in \underline{z X}(\mathbb U)$.
In particular, $\nabla(p)=z_1(p)$.
Now as $\nabla$ is induced by $e$ and $z_1$ is induced by $s_1=s$ and $\nabla(p) = z_1(p)$ (as both $\nabla(p)$ and $z_1(p)$ lie
on $z_1 X(\mathbb U)$ over $p$ but there is only one point in this fibre as $z_1$ is a morphism), we conclude that $e(p) = s(p)$.
That is, $p\in X(C)$.

For the converse, by Lemma~\ref{iterconstant} the maps $s_n$ and $E_n$ agree on $C$ for all $n$.  Thus, if $p \in X(C)$, then
$\nabla_n(p) = z_n(p)$ for all natural numbers $n$.  In particular, $\nabla_n(p) \in z_n X(\mathbb U)$ for all $n$ so that
$p \in \underline{z X}(\mathbb U)$.
\end{proof}

We show now that injectivity of a map of algebraic groups on the $C$-points is reflected by the injectivity of higher prolongations
of the map when restricted to the constant section.

\begin{lemma}
\label{eventualinjective}
Let $G$ be an algebraic group over $C$ and let $\Lambda:G_\UU \to H$ be a morphism of algebraic groups over
$\UU$ where $G_\UU$ is the base change of $G$ to $\UU$.  If the kernel of the restriction of $\Lambda$ to $G(C)$ is
trivial, then for $n \gg 0$, the kernel of the restriction of $\tau_n(\Lambda)$ to $z_n G(\UU)$ is trivial.
\end{lemma}
\begin{proof}
For each natural number $n$, define $K_n := \ker (\tau_n(\Lambda) \upharpoonright z_n G)$ as an algebraic subgroup
of $z_n G$.
We aim to show that the $K_n$'s are eventually trivial.
Define $V_n := \bigcap_{m \geq n} \widehat{f}_{m,n} (K_m)$ where $\widehat{f}_{m,n}:\tau_m G \to \tau_n G$ is
the map in the projective system defining the $\underline{\cD}$-variety $\underline{G} = (\tau_n G)$.   Since each $K_m$ is an
algebraic group and $\widehat{f}_{m,n}$ is a morphism of algebraic groups, $V_n$ is an algebraic subgroup of $z_n G$.
If the $K_n$ were nontrivial for arbitrarily large $n$, then because $\widehat{f}_{m,n}$ maps
$K_m$ to $K_n$, the $V_n$ would also be nontrivial.
So it suffices to show that $V_n=\{1\}$ for all $n$.

From the descending intersection defining $V_n$, note that $\widehat{f}_n:V_n \to V_0$ is surjective.
Thus,
$V_n = z_n V_0$.
We now prove that $\underline V=(V_n)$ is a $\sigma$-dominant $\underline\cD$-variety.
Note that if we knew  $V_0$ were defined over the constants, this would have followed automatically from our discussion of constant sections.

By definition, $K_{n+1} = \ker \tau_{n+1}(\Lambda) \upharpoonright z_{n+1} G$.
On the other hand, note that whenever one has a morphism of algebraic groups $\rho:A\to B$, then $\tau(\ker\rho)=\ker(\tau\rho)$.
Indeed, this follows from the fact that, for any algebra $R$, the identification of $\tau A(R)$ with $A\big(\cD^e(R)\big)$ identifies $\tau \rho$ with $\rho$ evaluated on $A\big(\cD^e(R)\big)$.
Hence $\tau K_n = \ker(\tau \tau_n(\Lambda)
\upharpoonright \tau z_n G)$.  Thus, from the inclusion $z_{n+1} G \hookrightarrow \tau z_n G$, we obtain an inclusion
$K_{n+1} \hookrightarrow \tau K_n$.  Taking intersections, we see that $V_{n+1}$ is included in $\tau V_n$.  Thus, $\underline{V} = (V_n)$ is a $\underline \cD$-subvariety of $\underline{zG}=(z_nG)$.
For $\sigma$-dominance, fixing $i=0,\dots,t$ and $n>0$, we have
\begin{eqnarray*}
V_{n-1}
&=&
\hat\pi_i\big(zV_{n-1}) \ \ \text{as $\hat\pi_i\circ z=\id$}\\
&=&
\hat\pi_i(V_n)\\
&\subseteq&
V_{n-1}^{\sigma_i} \ \ \ \text{as $\hat\pi_i:\tau(V_{n-1})\to V_{n-1}^{\sigma_i}$}
\end{eqnarray*}
And so we have equality throughout, and $\underline V$ is a $\sigma$-dominant $\underline\cD$-variety.

However,
$\underline{V}(\UU) \subseteq \underline{zG}(\UU) \cap \ker(\Lambda)(\UU) = G(C) \cap \ker(\Lambda)(\UU) = \{ 1 \}$, the penultimate equality being Lemma~\ref{zerosect}.
Thus, by Proposition~\ref{genpoints}, we must have $V_n(\UU)  = \{ 1 \}$ for all $n$, as desired.
\end{proof}

We are now in a position to show that the jet spaces of a finite dimensional $\underline{\cD}$-variety are always finite dimensional
$C$-vector spaces.

\begin{proposition}
\label{smalljet}
Suppose $\underline Z=(Z_n)$ is a $\underline\cD$-subvariety of $X$ over $k$ such that $\dim Z_n$ is bounded independently of $n$.
Suppose $p\in\underline Z(\mathbb U)$ is $k$-generic.
Then for each $m\geq 1$, $\jet^m_{\underline\cD}(\underline Z)_p(\mathbb U)$ is a finite dimensional $C$-vector space.
\end{proposition}

\begin{proof}
Fix $m\geq 1$.
As explained at the beginning of this section, see~(\ref{djet}) in particular, we have that $\jet^m_{\underline\cD}(\underline Z)_p=(T_n)_{n<\omega}$ where $T_n:=\phi\big(\jet^m(Z_n)_{\nabla_n(p)}\big)$.
As $\dim Z_n$ is bounded, so is $\dim(\jet^m(Z_n)_{\nabla_n(p)})$, and hence also $\dim T_n$.
Let $N$ be a bound on $\dim T_n$.
We will show that the $C$-dimension of $\jet^m_{\underline\cD}(\underline Z)_p(\mathbb U)$ is bounded by $N$.

Toward a contradiction, set $\mu:=N+1$ and suppose $\lambda_1,\dots,\lambda_\mu\in \jet^m_{\underline\cD}(\underline Z)_p(\mathbb U)$ are $C$-linearly independent.
Consider the map $g:\Ga^\mu\to \jet^m(X)_p$ given by $(x_1,\dots,x_\mu)\mapsto\sum_{i=1}^\mu x_i\lambda_i$.
Note that $g$ restricted to the $C$-points is injective, and hence Lemma~\ref{eventualinjective} will apply.
Recall that by Lemma~\ref{zerosect}, $\Ga^\mu(C)=\underline {z\Ga^\mu}(\mathbb U)$.
We claim that for each $n$, $\tau_ng$ restricts to a morphism from $z_n\Ga^\mu$ to $T_n$.
It suffices to check that $\tau_ng$ takes a $k$-generic point of $z_n\Ga^\mu$ to $T_n$.
But if we take a $k$-generic point of $\underline {z\Ga^\mu}$, say $q\in\Ga^\mu(C)$, then $\nabla_n(q)$ is $k$-generic in $z_n\Ga^\mu$, and
$\tau_ng\big(\nabla_n(q)\big) =\nabla_n\big(g(q)\big)$ by Proposition~4.7(a) of~\cite{paperA}.
This latter is in $T_n$ since $g(q)\in\jet^m_{\underline\cD}(\underline Z)_p(\mathbb U)$.
So $\tau_ng$ does restrict to a morphism from $z_n\Ga^\mu$ to $T_n$.
Lemma~\ref{eventualinjective} tells us that this morphism is an embedding, for sufficiently large $n$.
Hence eventually $T_n$ has dimension at least $\mu$, which is a contradiction.
\end{proof}

\medskip
\subsection{A canonical base property}
Using jet spaces we obtain a description of canonical bases and deduce therefrom the Zilber dichotomy for finite-dimensional rank one types.
Canonical bases for simple theories were introduced as hyperimaginary elements in~\cite{hartkimpillay}, to which we refer the reader for further details. ere we show that the canonical bases in $\ecdf$ are interalgebraic with infinite sequences of real elements.
As might be expected from our description of types in~$\S$\ref{subsection-types}, the canonical base of a type $\tp(a/L)$ will need to take into account not just the $L$-loci of the $\nabla_n(a)$, but indeed of the $\Theta_r(a)$, where recall that
$$\Theta_r(a):=\big(\theta a:\theta\text{ a word of length $\leq r$ on }\{\partial_1,\dots,\partial_\ell,\sigma_1^{-1},\dots,\sigma_t^{-1}\} \big).$$

\begin{theorem}
\label{cbjet}
Suppose $L$ is an algebraically closed inversive $\cD$-subfield and $a$ is a finite tuple from $\mathbb U$.
Let $\underline Z:=\dlocus(a/L)$ in the sense that $Z_n=\locus(\nabla_n a/L)$ for all $n<\omega$.
For each $r<\omega$, let $\Theta_r\underline Z:=\big(\locus(\nabla_n\Theta_r a)/L)\big)_{n<\omega}$.
Then
$$\cb(a/L)\subseteq\acl\left(\{a\}\cup\bigcup_{m\geq 1,r\geq 0}\jet_{\underline\cD}^m(\Theta_r\underline Z)_{\Theta_ra}(\mathbb U)\right)$$
If $\tp(a/L)$ is finite-dimensional then
$$\cb(a/L)\subseteq\acl\left(\{a\}\cup\bigcup_{m\geq 1,r\geq 0}\jet_{\underline\cD}^m(\nabla_r\underline Z)_{\nabla_r(a)}(\mathbb U)\right)$$
\end{theorem}

\begin{proof}
Note that $\underline\cD$-locus of a tuple over an algebraically closed inversive $\cD$-field will always be a $\sigma$-dominant absolutely irreducible $\underline\cD$-subvariety.
Hence this is the case for $\underline Z$, $\Theta_r\underline Z$, and $\nabla_r\underline Z$.

Let $K\subseteq L$ be the inversive $\cD$-subfield generated by the minimal fields of definition of all the $\locus(\Theta_ra/L)$, as $r<\omega$ varies.

\begin{claim}
\label{cb}
$\cb(a/L)\subseteq\acl(K)$ and $K\subseteq\dcl\big(\cb(a/L)\big)$
\end{claim}
\begin{proof}[Proof of Claim~\ref{cb}]
By Lemma~\ref{dimfork} we have that $a\ind_KL$.
We also know that $\tp(a/K^{\alg})$ is an amalgamation base because the independence theorem holds over algebraically closed sets.
It follows that $\cb(a/L)$ is in the definable closure of $K^{\alg}$ and hence in the algebraic closure of $K$.

For the other containment, let $L_0=\dcl\big(\cb(a/L)\big)$.
Then, as $a\ind_{L_0}L$, we have that $\locus(\Theta_ra/L)=\locus(\Theta_ra/L_0)$, for each $r<\omega$.
Hence the minimal field of definition of each $\locus(\Theta_ra/L)$ is a subfield of $L_0$.
It follows that $K$, which is in the definable closure of these minimal fields of definition, must also be contained in $L_0$, as desired.
, so $K\subseteq L_0$.
\end{proof}

Given the above claim it suffices to show that $K$ is in the definable closure of $\displaystyle \{a\}\cup\bigcup_{m\geq 1,r\geq 0}\jet_{\underline\cD}^m(\Theta_r\underline Z)_{\Theta_ra}(\mathbb U)$.
Suppose that $\alpha$ is an automorphism that fixes $a$ and all the $\jet_{\underline\cD}^m(\Theta_r\underline Z)_{\Theta_ra}(\mathbb U)$ pointwise.
Then $\alpha$ takes $L$ to $L':=\alpha(L)$ and it takes $\Theta_r\underline Z$ to the $\underline\cD$-suvariety $\underline Y^{(r)}:=\big(\locus(\nabla_n\Theta_r a)/L')\big)_{n<\omega}$.
Note that $\Theta_ra$ is generic in $\Theta_r\underline Z$ over $L$ and in $\underline Y^{(r)}$ over $L'$.
Note also that
$$\jet_{\underline\cD}^m(\nabla_s\Theta_r\underline Z)_{\nabla_s(\Theta_ra)}(\mathbb U)=\jet_{\underline\cD}^m(\nabla_s\underline Y^{(r)})_{\nabla_s(\Theta_ra)}(\mathbb U)$$
for all $m\geq 1$ and $s\geq 0$.
Indeed, there is a co-ordinate projection taking $\Theta_{r+s}(a)$ to $\nabla_s(\Theta_ra)$ that will induce a definable surjection from $\jet_{\underline\cD}^m(\Theta_{r+s}\underline Z)_{\Theta_{r+s}(a)}(\mathbb U)$ to $\jet_{\underline\cD}^m(\nabla_s\Theta_r\underline Z)_{\nabla_s(\Theta_ra)}(\mathbb U)$, and since the former is fixed pointwise by $\alpha$, so is the latter.
It follows by Fact~\ref{jetsdetermine}, then, that $\Theta_r\underline Z=\underline Y^{(r)}$.
In particular, $\alpha$ fixes $\locus(\Theta_ra/L)$ for all $r\geq 0$.
Hence $\alpha\upharpoonright_K=\id$, as desired.

In the finite-dimensional case we can carry out the same argument but with $\nabla_r$ instead of $\Theta_r$.
That is, we set $K$ to be the inversive $\cD$-field generated by the minimal fields of definition of $\locus(\nabla_ra/L)$, as $r$ varies.
Using Corollary~\ref{zoecor} instead of Lemma~\ref{dimfork}, the above argument goes through with this $K$, and we get the desired description of the canonical base in the finite-dimensional case.
\end{proof}

The following is a generalisation of Theorems~1.1 and~1.2 of~\cite{pillayziegler03}.

\begin{corollary}[The canonical base property for finite-dimensional types]
\label{cbp}
Suppose $\tp(a/k)$ is finite-dimensional and $L$ is an algebraically closed inversive $\cD$-field extending $k$ such that $\cb(a/L)$ is interalgebraic with $L$ over $k$.
Then $\tp(L/k\langle a\rangle)$ is almost internal to the constants.
\end{corollary}

\begin{proof}
Let $\underline X:=\dlocus(a/k)$ and $\underline Z:=\dlocus(a/L)$.
The finite-dimensionality of $\tp(a/k)$ implies in particular that for each $r<\omega$, $\dim\locus(\nabla_n\nabla_ra/k)$ is bounded independently of $n$, and so by Proposition~\ref{smalljet}, for each $m\geq 1$, $\jet^m_{\underline\cD}(\nabla_r\underline X)_{\nabla_r(a)}(\mathbb U)$ is a finite dimensional $C$-vector space.
Let $B$ be a countable set that contains a $C$-basis for $\jet^m_{\underline\cD}(\nabla_r\underline X)_{\nabla_r(a)}(\mathbb U)$, for all $m$ and $r$.
As these jet  spaces are type-definable over $k\langle a\rangle$, we can choose $B$ so that $B\ind_{k\langle a\rangle}L$.
Then
\begin{eqnarray*}
L
&\subseteq&
\acl\left(k\langle a\rangle\cup\bigcup_{m\geq 1,r\geq 0}\jet_{\underline\cD}^m(\nabla_r\underline Z)_{\nabla_ra}(\mathbb U)\right) \ \ \ \ \ \ \text{ by Theorem~\ref{cbjet}}\\
&\subseteq&
\acl\left(k\langle a\rangle\cup B\cup C\right) \ \ \ \ \ \ \text{ as the $\jet^m_{\underline\cD}(\nabla_r\underline Z)_{\nabla_r(a)}(\mathbb U) \ \leq \ \jet^m_{\underline\cD}(\nabla_r\underline X)_{\nabla_r(a)}(\mathbb U)$}
\end{eqnarray*}
So $\tp(L/k\langle a\rangle)$ is almost internal to the constant field $C$, as desired.
\end{proof}

\begin{corollary}[The Zilber dichotomy for finite-dimensional types]
\label{dichotomy}
Suppose $p$ is a finite-dimensional $\operatorname{SU}$-rank one type over a substructure of some model of $\ecdf$.
Then $p$ is either one-based or almost internal to the constants.
\end{corollary}
\begin{proof}
Work in a saturated model $(\mathbb U,\partial)$ of $\ecdf$, and suppose $p=\tp(a/k)$ for some inversive $\cD$-field $k$.
If $p$ is not one-based then there exists an algebraically closed inversive $\cD$-field $L$ extending $k$ such that $\cb(a/L)$ is interalgebraic with $L$ over $k$, but $\tp(L/ka)$ is nonalgebraic.
Since $p$ is finite-dimensional the CBP (Corollary~\ref{cbp}) applies and we have that $\tp(L/ka)$ is almost $C$-internal.
So $L\subseteq \acl(kBac)$ where $L\ind_{ka}B$ and $c$ is a tuple from $C$.
Hence $L\nind_{kBa}c$.
On the other hand, being (interalgebraic over $k$ with) the canonical base, $L\subseteq \acl(ka_1\dots a_n)$ for some independent realisations $a_1,\dots, a_n$ of $p$,
which we may assume to be independent of $Ba$ over $k$.
So $(a_1,\dots,a_n)\nind_{kBa}c$, and hence for some $i<n$, $a_{i+1}\in\acl(kBaa_1\dots a_ic)$.
But as $a_{i+1}\ind_kBaa_1\dots a_i$ by choice,
this witnesses that $p$ is almost internal to $C$.
\end{proof}

\begin{remark}
We do not know if the restriction of the above corollary to finite-dimensional types is necessary.
In the case of partial differential fields, by which we mean differentially closed field of
characteristic zero for finitely many commuting derivations, it follows from the analysis of regular non one based types in~\cite{MPS} that those of rank one are in fact finite-dimensional.
(This was observed first in the difference-differential case by Bustamante~\cite{bustamante}.)
We expect the same to hold here.
\end{remark}

\bigskip
\section{Appendix: On the assumptions}

\noindent
In proving the existence of the model companion we restricted ourselves to characteristic zero, and we also imposed on $A$ the properties described in
Assumptions~\ref{assumptionA}.
In this appendix, we discuss  the extent to which these restrictions are necessary.

\medskip
\subsection{No model companion in positive characteristic}
To begin with, in most cases the restriction to characteristic zero is necessary.
While model companions are known to exist in positive characteristic for the differential and difference cases,
at the level of generality considered in this paper model companions do not necessarily exist in characteristic $p>0$.
We will prove this by showing that the condition of having a $p$th root in some $\cD$-field extension is not in general first-order.

For the time being Assumptions~\ref{assumptionA} remain in place.

\begin{proposition}
\label{pthrootcriterion}
Suppose $(K,\partial)$ is a $\cD$-field of characteristic $p>0$ and $a\in K$.
Then the following are equivalent:
\begin{itemize}
\item[(i)]
There is a $\cD$-field extension of $K$ in which $a$ has a $p$th root.
\item[(ii)]
For each $n<\omega$, $E_n(a)\in\cD_n^p(K)$.
Here $\cD_n^p$ is the $\mathbb S$-algebra scheme that is the image of $\cD_n$ under the $\mathbb S$-algebra morphism of raising to the power $p$.
\end{itemize}
\end{proposition}

\begin{proof}
First of all, note that $E_n(a)\in\cD_n^p(K)$ if and only if $E_n(a)$ is a $p$th power in $\cD_n(K^{\alg})$, if and only if $E_n(a)$ is a $p$th power in $\cD_n(L)$ for some field extension $L$ of $K$.

(i)$\implies$(ii).
Let $(L,\partial)$ extend $(K,\partial)$ with $b\in L$ such that $b^p=a$.
Then
$$E_n(b)^p=E_n(b^p)=E_n(a)$$
in $\cD(L)$ showing that $E_n(a)$ is a $p$th power for each $n<\omega$, as desired.

(ii)$\implies$(i).
Passing to an elementary extension if necessary, we may assume that $K$ is $\aleph_0$-saturated as a field.
Now for each $n$ there exists $b_n\in\cD_n(K^{\alg})$ such that $b_n^p=E_n(a)$.
Since the minimal polynomial of $b_0$ over $K$ is $x^p-a$, we can extend $E_n:K\to\cD_n(K)$ to an $A$-algebra homomorphism $E_n':K(b_0)\to\cD_n(K^{\alg})$ by setting $E_n'(b_0):=b_n$.

Now, note then that $f_n(b_n)^p=E_{n-1}(a)$, and so by  $\aleph_0$-saturation, we may assume that $f_n(b_n)=b_{n-1}$ for each $n>0$.
In other words, there exists $\{b^{(\alpha)}:\alpha\in L_{<\omega}\}\subseteq K^{\alg}$ such that
$$b_n=\sum_{\alpha\in L_n}b^{(\alpha)}\epsilon_\alpha$$
where $(\epsilon_{\alpha}:\alpha\in L_n)$ is the basis for $\cD_n(A)$ over $A$ induced by $\psi_n$ from the standard basis for $A^{L_n}$.
Let $L=K\big(b^{(\alpha)}\big)_{\alpha\in L_{<\omega}}$ and extend $e:K\to\cD(K)$ to $L$ by
$$e'(b^{(\alpha)}):=\sum_{j=0}^{\ell-1}b^{(j^{\smallfrown}\alpha)}\epsilon_j$$
That this is an $A$-algebra  homomorphism extending $e$ follows from the fact that the $E_n'$ defined above were $A$-algebra homomorphisms extending $E_n$.
That $\pi\circ e=\id$ is clear from construction.
We have thus given $L$ a $\cD$-field structure extending $(K,\partial)$, and we have a $p$th root of $a$ in $L$, namely $b_0$.
\end{proof}

Using Proposition~\ref{pthrootcriterion} we show that under a weak hypothesis on $\cD(A)$, the theory of $\cD$-fields does not have a model companion.

\begin{proposition}
\label{nomcp}
If $p$ is a prime and there is some nilpotent $\epsilon \in \cD(A)$ with $\epsilon^p \neq 0$,
then the theory of $\cD$-fields of characteristic $p$ does not have a model companion.

In particular, the class of fields of characteristic $p$ equipped with a truncated higher derivation of length greater than $p$ does not have a model companion--  see Example~\ref{examples}(b).
\end{proposition}

\begin{proof}
Suppose the theory
of $\cD$-fields does have a model companion, $T$, and seek a contradiction.
By Proposition~\ref{pthrootcriterion}, in an
existentially closed $\cD$-field the partial type $\Phi(x):= \{ E_n(x) \in \cD_n^p \}_{n=0}^\infty$ is equivalent to the
formula $(\exists y) y^p = x$.
As every model of $T$ is existentially closed, this equivalence is entailed by $T$.
We get by compactness that there is some $\ell\geq 0$ such that in every model of $T$, $\bigwedge_{n=0}^\ell(E_n(x) \in \cD_n^p)$ implies $\Phi(x)$.
But as these formulas are quantifier-free, and every $\cD$-field embeds into a model of $T$, we have that in every $\cD$-field $(L,\partial)$
\begin{equation}
\label{bound}
\bigwedge_{n=0}^\ell E_n(x) \in \cD_n^p(L) \ \implies \ \bigwedge_{n=0}^\infty E_n(x) \in \cD_n^p(L)
\end{equation}
We will construct a counterexample to this claim.

Since $\cD$ is an $\mathbb S$-algebra, it is itself of characteristic $p$.  Hence, the Frobenius defines a morphism of
ring schemes $F:\cD \to \cD$ given on points as $F:\cD(R) \to \cD(R)$ via $a \mapsto a^p$ where the $p^\text{th}$
power is taken with respect to the multiplication in $\cD$.  Let ${\mathcal N}$ be the nilradical of
$\cD$ considered as a subgroup scheme of $(\cD,+)$.  Visibly, ${\mathcal N}$ is mapped back to itself by $F$.
Since the nilradical of $\cD(A)$ is nontrivial, the kernel of $F:{\mathcal N} \to {\mathcal N}$
has positive dimension. Hence, its image has dimension strictly less than $\dim {\mathcal N}$.  Let
$\eta \in {\mathcal N}(A) \smallsetminus  \cD(A^\text{alg})^p$.  

Let $m=\ell+1$ and let $L := A(x_1,\ldots,x_m)$
be the field of rational functions in $m$ variables over $A$.  Define $e:A[x_1,\ldots,x_m] \to \cD(L)$ by
$e(x_i) := x_i + x_{i+1} \epsilon^p$ for $i < m$ and $e(x_m) := x_m + \eta$.  Since
$\eta$ and $\epsilon$ are nilpotent, we see that the composition of $e$ with the
reduction map $\cD(L) \to \cD(L)/{\mathcal N}(L)$ may be identified with the standard
algebra structure map.  Hence, each of the associated endomorphisms  is simply the identity and
we may therefore extend $e$ to $L$ to give $L$ a $\cD$-field structure.
For each natural number $n$, let $E_n:L \to \cD_n(L)$ be the map obtained from $e$ by iteration.   

We show now that for $n < m$ that $E_n(x_1) \in \cD_n^p(L)$ but
$E_m(x_1) \notin \cD_m^p(L)$. This fact will contradict~(\ref{bound}).

We need some notation.  For any natural number $n$ and 
set $\tau \subset \{1, \ldots, n \}$, we define 
$b_\tau := \bigotimes_{i=1}^n \epsilon^{\tau(i)} \in \cD^{(n)}(A)$ where here 
we have identified $\tau$ with its characteristic function.  Let us note
that if we fix $j \leq n$, then 
$\sum_{\tau \in \binom{n}{j}} b_\tau \in \cD_n(A)$.

Working inductively, we see that 
for $n < m$, we have 
$$
E_n(x_1) = \sum_{j=0}^n  x_{j+1} \sum_{\tau \in \binom{n}{j}} b_\tau^p
$$
which is clearly an element of $\cD_n^p(L)$, that is $E_n(x_1) \in \cD_n(L^{\alg})^p$.
Indeed, for the base case of the induction we have
$E_1(x_1) = x_1 + x_2 \epsilon^p = x_1 b_{\varnothing}^p + x_2 b_{ \{ 1 \}}^p$ and
for $n+1 < m$ we have the following computation. 
\begin{eqnarray*}
E_{n+1}(x_1) & = &  \sum_{j=0}^n x_{j+1} \sum_{\tau \in \binom{n}{j}} b_\tau^p \otimes 1 
+ \sum_{j=0}^n x_{j+2} \sum_{\tau \in \binom{n}{j}} b_\tau^p \otimes \epsilon^p \\
& = & \sum_{\ell=0}^{n+1} x_{\ell+1}  \sum_{\tau \in \binom{n+1}{\ell}} b_\tau^p
\end{eqnarray*}

On the other hand, evaluating $E_m$, 
we obtain 
\begin{eqnarray*}
E_m(x_1)  & = & \sum_{j=0}^{m-2} (x_{j+1}  \sum_{\tau \in \binom{m}{j}} b_\tau^p \otimes 1 + x_{j+2} \sum_{\tau \in \binom{m}{j}} b_\tau^p \otimes \epsilon^p) \\ &&
 + x_m (\epsilon^p \otimes \cdots \otimes \epsilon^p \otimes 1) + (\epsilon^p \otimes \cdots \otimes \epsilon^p \otimes \eta)
\end{eqnarray*}
In this expression, every term other than $\epsilon^p \otimes \cdots \epsilon^p \otimes \eta$ belongs to $\cD_m^p(L)$
while $\epsilon^p \otimes \cdots \otimes \epsilon^p \otimes \eta$ does not.  Hence, $E_m(x_1) \notin \cD_m^p(L)$. 
\end{proof}

\medskip
\subsection{Removing Assumptions~\ref{assumptionA}}
\label{subsect-assumptionA}
On the other hand, if we restrict to characteristic zero, then model companions  exist even in the absence of
Assumptions~\ref{assumptionA}.
As we do not yet see a pressing reason to develop the theory in full generality, we restrict ourselves here to a sketch of a proof.

We explain first why  Assumption~\ref{assumptionA}(ii) is unnecessary.
That is, we will describe the model companion still assuming that $A$ is a field of characteristic zero, but {\em without} assuming in the decomposition $\cD(A)=\prod_{i=0}^tB_i$ that the residue field of each $B_i$ is $A$.

For each $i$ fix an irreducible polynomial $P_i(x)$ of degree $d_i$ such that the residue field of $B_i$ is the finite extension $A[x]/(P_i)$.
Denote by $\cE_i$ the $\bbs$-algebra scheme such that $\cE_i(R)=R[x]/(P_i)$ for any $A$-algebra $R$.
In particular this fixes a basis for $\cE_i$.
Note that we may assume $\cE_0=\bbs$; indeed, one of the $B_i$ will still correspond to the kernel of $\pi$ and hence will have residue field $A$.

We have as before $\cD=\prod_{i=0}^t\cD_i$ and $\theta_i:\cD\to\cD_i$, but now $\rho_i:\cD_i\to\cE_i$ are the $\bbs$-algebra homomorphisms which when evaluated at $A$ are the residue maps on  $B_i$.
Note that when evaluated on another $A$-algebra $R$, even if $R$ is a field extension, $\cD_i(R)$ need no longer be a local ring and $\cE_i(R)$ may no longer be a field.
Nevertheless, $\rho_i^R:\cD_i(R)\to\cE_i(R)$ will be a surjective ring homomorphism, it is obtained from the residue map $\rho_i^A$ by base change to $R$.
As before we set $\pi_i:=\rho_i\circ\theta_i:\cD\to\cE_i$.

Suppose we are given a $\cD$-ring $(R,\partial)$.
For each $i=0,\dots, t$, instead of an associated endomorphism we now only have the {\em associated $A$-algebra homomorphisms} $\sigma_i:=\pi_i^R\circ e:R\to\cE_i(R)$, which with  respect to the basis for $\cE_i$ fixed above can be written as
$$\sigma_i(a)= \sum_{j=0}^{d_i-1}\alpha_{ij}(a)x^j$$
The $\alpha_{ij}:R\to R$ will be $A$-linear maps that are $0$-definable in $(R,\partial)$; indeed they are fixed $A$-linear combinations of the original operators $\partial$.
(We are working in the language $\cL_{\cD}$ of $A$-algebras equipped with the operators $\partial_1,\dots,\partial_{\ell-1}$, see page~\pageref{language}.)
Note that $d_0=1$, $\pi_0=\pi$, and $\sigma_0=\alpha_{0,0}=\id$.

Extending our earlier notation we now let $\cK$ be the class of $\cD$-rings $(R,\partial)$ such that $R$ is an integral $A$-algebra and for each $i=1,\dots,t$, $\sigma_i:R\to\cE_i(R)$  is {\em injective and has no zero divisors in its image}.
Notice that as $\cE_i(R)$ need not be an integral domain, this latter constraint on the $\sigma_i$ is not vacuous.
The class $\cK$ is universally axiomatisable; this follows from the fact that the $\alpha_{ij}$ are quantifier-free definable, as is the ring structure on $R^{d_i}$ induced by the ring structure on $\cE_i(R)$ via the basis $\{1,x,\cdots,x^{d_i-1}\}$.

More generally,
the class $\cM$ is now the class of triples $(R,S,\partial)$ where $R\subseteq S$ are integral $A$-algebras and $\partial=(\partial_1,\dots,\partial_{\ell-1})$ is a sequence of maps from $R$ to $S$ such that $e:R\to\cD(S)$ given by $e(r):=r\epsilon_0+\partial_1(r)\epsilon_1+\cdots+\partial_{\ell-1}(r)\epsilon_{\ell-1}$ has the following properties:
\begin{itemize}
\item[(i)]
$e$ is an $A$-algebra homomorphism,
\item[(ii)]
 for each $i=1,\dots, t$, $\sigma_i:=\pi_i^S\circ e:R\to \cE_i(S)$ is injective and has no zero divisors in its image.
\end{itemize}
Note that $\sigma_0:=\pi_0^S\circ e=\pi^S\circ e:R\to S$ is then the identity on $R$.

\begin{lemma}
\label{extendtoalg-gen}
\begin{itemize}
\item[(a)]If $(R,L,\partial)\in\cM$ with $L$ a field, then we can (uniquely) extend $\partial$ to the fraction field $F$ of $R$ so that $(F,L,\partial)\in\cM$.
\item[(b)]
Suppose $(F,L,\partial)\in\cM$ where $F$ and $L$ are fields and $L$ is algebraically closed.
If $\sigma_i:F\to \cE_i(L)$ are the embeddings associated to $\partial$, and $\sigma_i'$ are extensions of $\sigma_i$ to $F^{\alg}$, then there is an extension $\partial'$ of $\partial$ to $F^{\alg}$ with associated embeddings $\sigma_i'$.
\end{itemize}
\end{lemma}

\begin{proof}
Part~(a) is proved along the lines of Lemma~\ref{extendtofields}.
In order to extend $e$ to the fraction field $F$ we need to show that $e$ takes nonzero elements of $R$ to units in $\cD(L)$.
Equivalently, for each $i=1,\dots,t$, we need to show that $\theta_i\circ e$ takes nonzero elements of $R$ to units in $\cD_i(L)$.
Note that as $L$ is a field, $\cD_i(L)$ is a product of local $A$-algebras and $\cE_i(L)$ is the product of the residue fields of these local $A$-algebras.
The units of $\cD_i(L)$ are therefore precisely those elements whose images in $\cE_i(L)$ under $\rho_i$ are neither zero nor zero divisors.
What we therefore need to verify is that $\rho_i\circ\theta_i\circ e$ is injective and has no zero divisors in its image.
But $\rho_i\circ\theta_i\circ e=\sigma_i$ and the desired property is true since $(R,L,\partial)\in\cM$.
That the corresponding extension $(F,L,\partial)$ lands back inside $\cM$ is clear.

For part~(b) it suffices to prove, as in the proof of Lemma~\ref{extendtoalg}, that for any $a\in F^{\alg}$ we can
extend~$e$ to $F(a)$ in such a way that $\pi_i^L\circ e(a)=\sigma_i'(a)$ for each $i=0,\dots,t$.
Let $P(x)\in F[x]$ be the minimal polynomial of $a$ over $F$
and let $c_i:=\sigma_i'(a)\in\cE_i(L)$.
Note that as $P(a)=0$ but $\frac{d}{dx}P(a)\neq 0$ in $F^{\alg}$, and since $\sigma_i'$ is a ring homomorphism, we have that $P^{\sigma_i}(c_i)=0$ while $\frac{d}{dx}P^{\sigma_i}(c_i)$ is a unit in $\cE_i(L)$.
We wish to lift this root to $\cD_i(L)$.
While it is not the case that $\rho^L_i:\cD_i(L)\to\cE_i(L)$ is the residue map of a local algebra, it is still surjective with nilpotent kernel.
This is because the kernel of $\rho_i^L$ is obtained from the kernel of $\rho_i^A$ by tensoring with $L$ over $A$, and the kernel of the latter is the maximal ideal of $B_i$ which is nilpotent.
Hence by a Hensel's Lemma type argument we can lift $c_i$ to a root $b_i$ of $P^{e_i}(x)$ in $\cD_i(L)$.
Then $b=(b_0,\dots,b_t)\in\cD(L)$ is a root of $P^e(x)\in\cD(L)[x]$, and we can extend $e$ to $F(a)$ by sending $a$ to $b$.
By construction $\pi_i^Le(a)=\sigma_i'(a)$.
\end{proof}

Suppose now that $(K,\partial)$ is an algebraically closed $\cD$-field.
Recall that we wrote each $B_i=A[x]/(P_i)$.
 Let $b_{i1},\dots,b_{id_i}$ be the distinct roots of $P_i$ in $K$.
Then using these roots we can decompose $\cE_i(K)$ into a
 power of $K$ itself:
$$\cE_i(K)=K[x]/(P_i)=\prod_{k=1}^{d_i}K[x]/(x-b_{ik})=K^{d_i}$$
Composing the associated homomorphism $\sigma_i$ with the co-ordinate projections we get a $d_i$-tuple of {\em associated endomorphisms} of $K$, $(\sigma_{i1},\dots,\sigma_{id_i})$ where
$$\displaystyle \sigma_{ik}:= \sum_{j=0}^{d_i-1}b_{ik}^j\alpha_{ij}$$
In fact, under the identification $\cE_i(K)=K^{d_i}$ we have $\sigma_i=(\sigma_{i1},\dots,\sigma_{id_i})$, and so by the {\em associated difference field $(K,\sigma)$} we mean the the field $K$ equipped with all of these endomorphisms.
It then also makes sense to say that $(K,\partial)$ is {\em inversive} if each $\sigma_{ik}$ is an automorphism.
It is important to note, though, that the $\sigma_{ik}$, while still definable in $(K,\partial)$, are now not $0$-definable but $b_{ik}$-definable.
Note also that we have only defined the associated difference field of an {\em algebraically closed} $\cD$-field.

\begin{example}
\label{no4ii2}
Consider Example~\ref{no4ii} where
$$A=\mathbb Q\text{ and }\cD(A)=\mathbb Q \times \mathbb Q[x]/(x^2-2)$$
Then $t=1$ and the homomorphism associated to a $\cD$-ring $(R,\partial_1,\partial_2)$ is $\sigma_1:R\to R[x]/(x^2-2)$ where $\sigma_1(a)=\partial_1(a)+\partial_2(a)x$.
So $\alpha_{10}=\partial_1$ and $\alpha_{11}=\partial_2$.
If $(K,\partial_1,\partial_2)$ is an algebraically closed $\cD$-field then the associated endomorphisms are $\sigma_{10}=\partial_1+\sqrt{2}\partial_2$ and $\sigma_{11}=\partial_1-\sqrt{2}\partial_2$.
\end{example}

\begin{lemma}
\label{extendtoautos-gen}
Suppose $(F,L,\partial)\in\cM$ where $F$ and $L$ are fields.
Then there exists an algebraically closed extension $K$ of $L$ and an extension of $\partial$ to an inversive $\cD$-field structure on $K$.
\end{lemma}

\begin{proof}
Replacing $L$ with $L^{\alg}$ we may assume that $L$ is algebraically closed.
We can thus write $\sigma_i=(\sigma_{i1},\dots,\sigma_{id_i})$ where the embeddings $\sigma_{ik}:F\to L$ are obtained by composing $\sigma_i:F\to\cE_i(L)$ with the $k$th projection in the decomposition $\cE_i(L)=L^{d_i}$.
We can extend these $\sigma_{ik}$ to automorphisms $\sigma_{ik}'$ of some algebraically closed $K\supseteq L$.
So $\sigma_i':=(\sigma_{i1}',\dots,\sigma_{id_i}'):K\to\cE_i(K)$ extends $\sigma_i$.
Now, as in Lemma~\ref{extendtoautos}, we fix a transcendence basis $B$ for $K$ over $F$ and easily extend $\partial$ to $F[B]$ so that $(F[B],K,\partial)\in\cM$ and the associated homomorphisms are $\sigma_i'\upharpoonright F[B]$.
By Lemma~\ref{extendtoalg-gen}(a) we can extend $\partial$ to $F(B)$ preserving this property.
By Lemma~\ref{extendtoalg-gen}(b) we can extend $\partial$ further to a $\cD$-structure on $K=F(B)^{\alg}$ in such a way that the associated homomorphisms are $\sigma_1',\dots,\sigma_t'$, and hence the associated endomorphisms are the automorphisms $\sigma_{ik}'$.
\end{proof}

Suppose now that  $X$ is an irreducible affine variety over an algebraically closed $\cD$-field $K$.
For each $i=0,\dots,t$, we have the abstract prolongations with the induced morphisms as constructed in $\S4$ of~\cite{paperA}:
$$
\xymatrix{
\tau(X,\cD,e)\ar[rr]^{\widehat{\theta_i}} && \tau(X,\cD_i,e_i)\ar[rr]^{\widehat{\rho_i}} && \tau(X,\cE_i,\sigma_i)
}$$
But we also have, for each fixed $k=1,\dots,d_i$, the morphism
$$
\xymatrix{
\tau(X,\cE_i,\sigma_i)\ar[rr] && X^{\sigma_{ik}}
}$$
Indeed, the $k$th factor projection $\cE_i(K)=K^{d_i}\to K$ induces a map from
$X\big(\cE_i(K)\big)$, where $X$ is viewed as a scheme over $\cE_i(K)$ via base change coming from $\sigma_i:K\to\cE_i(K)$, to $X^{\sigma_{ik}}(K)$.\footnote{This morphism does not follow formally from the comparing of prolongations done in~\cite{paperA} because the identification $\cE_i(K)=K^{d_i}$ is over $B$, not $A$.
To fit into the formalism of~\cite{paperA} we would thus require $\sigma_i$ to be a $B_i$-algebra homomorphism, which it need not be as it may move the roots of~$P_i$.}
Composing, we have for each $i$ and $k$ the morphism
$$
\xymatrix{
\tau(X,\cD,e)\ar[rr]^{\widehat{\pi_{ik}}} && X^{\sigma_{ik}}
}$$
Moreover, if we set $F:=A(b_{ik})_{1\leq i\leq t, 1\leq k\leq d_i}$, and $X$ moves uniformly within an $F$-definable family of varieties, then so do the $\widehat{\pi_{ik}}:\tau(X,\cD,e)\to X^{\sigma_{ik}}$.

Now we can state the version of Theorem~\ref{ecedomains} without Assumption~\ref{assumptionA}(ii).

\begin{theorem}
\label{ecedomains-gen2}
Drop Assumptions~\ref{assumptionA}, and assume only that $A$ is a field of characteristic zero.
Then $(K,\partial)\in\cK$ is existentially closed if and only if
\begin{itemize}
\item[I.]
$K$ is an algebraically closed field.
\item[II.]
There exist distinct roots $b_{i,1},\dots,b_{i,d_i}$ of $P_i$ in $K$ such that $\displaystyle \sigma_{ik}:= \sum_{j=0}^{d_i-1}b_{ik}^j\alpha_{ij}$ is an automorphism of $K$,  for all $i=1,\dots,t$ and $k=1,\dots,d_i$.
\item[III.]
There exist distinct roots $b_{i,1},\dots,b_{i,d_i}$ of $P_i$ in $K$ such that if $X$ is an irreducible affine variety over $K$ and $Y\subseteq\tau(X,\cD,e)$ is an irreducible subvariety over $K$ such that $\widehat{\pi_{ik}}(Y)$ is Zariski dense in $X^{\sigma_{ik}}$ for all $i=0,\dots,t$ and $k=1,\dots,d_i$, then there exists $a\in X(K)$ with $\nabla(a)\in Y(K)$.
\end{itemize}
The theory of $\cD$-fields of characteristic zero thus admits a model companion.
\end{theorem}

Theorem~\ref{ecedomains-gen2} is proved just as Theorem~\ref{ecedomains} was, using Lemmas~\ref{extendtoalg-gen}(a), \ref{extendtoalg-gen}(b), and~\ref{extendtoautos-gen} in place of~\ref{extendtofields}, \ref{extendtoalg}, and~\ref{extendtoautos}.
That the given axioms are first-order also follows as before.
We omit the details.

On the face of it, Assumption~\ref{assumptionA}(i) is more serious than Assumption~\ref{assumptionA}(ii), but we may reduce
to the case where it holds.  Indeed, if $(K,\partial)$ is a $\cD$-field for which $A$ is not necessarily a field, then
by regarding $\cD$ as a ring scheme over the field of fractions of the image of $A$ in $K$, we may see $(K,\partial)$
as a $\cD$-field in which Assumption~\ref{assumptionA}(i) holds.  That is, if we consider each possible way in which
$\cD(K)$ may split as a product of local rings via maps defined by linear equations defined over the algebraic
closure of the field of fractions of the image of $A$, then we see that in every $\cD$-field one of these splittings must hold.
We obtain an axiomatisation in the absence of Assumptions~\ref{assumptionA} by taking each such possible form of the linear
maps used for a splitting as an antecedent and then relativising the axiomatisation of Theorem~\ref{ecedomains-gen2}.

%\bibliographystyle{plain}
%\bibliography{principal}

\end{document}